\documentclass[reqno, 10pt]{article}
\usepackage{amssymb, amsmath, amsthm, amsfonts, amscd, epsfig}
\usepackage{color}

\usepackage[english]{babel}
\usepackage{bbm}
\usepackage{graphicx, epsfig}
\usepackage{geometry,graphicx,pgfplots,caption}
\usepackage{verbatim}
\usepackage{dsfont}
\usepackage{bm}

\usepackage[titletoc,toc,title]{appendix}
\usepackage[normalem]{ulem}
\newtheorem{theorem}{Theorem}[section]

\newtheorem{lemma}{Lemma}[section]

\newtheorem{definition}{Definition}[section]

\newtheorem{ass}{Assumption}[section]

\newtheorem{remark}{Remark}[section]
{\bf}{\it}

\newcommand{\Om}{\Omega}

\newcommand{\p}{\partial}

\newcommand{\beq}{\begin{equation}}
\newcommand{\eeq}{\end{equation}}
\newcommand{\RN}[1]{%
\textup{\uppercase\expandafter{\romannumeral#1}}%
}

\usepackage{lineno}

\renewcommand{\d}{\,\mathrm{d}}

\numberwithin{equation}{section}
\numberwithin{figure}{section}

\begin{document}

\title{Unique Determination of Variable Order in Subdiffusion from a Single Measurement\thanks{B. Jin is partly supported by Hong Kong RGC General Research Fund (14306824) and ANR / Hong Kong RGC Joint Research Scheme (A-CUHK402/24) and a start-up fund from The Chinese University of Hong Kong. The work of Y. Kian is supported by the French National Research Agency ANR and Hong Kong RGC Joint
Research Scheme for the project IdiAnoDiff (grant ANR-24-CE40-7039).}}

\author{Jiho Hong\thanks{Department of Mathematics, The Chinese University of Hong Kong, Shatin, N. T., Hong Kong SAR, P.R. China (\texttt{jihohong@cuhk.edu.hk, b.jin@cuhk.edu.hk})}\and Bangti Jin\footnotemark[2]\and Yavar Kian\thanks{Univ Rouen Normandie, CNRS, Normandie Univ, LMRS UMR 6085, F-76000 Rouen, France (\texttt{yavar.kian@univ-rouen.fr})}}

\maketitle

\begin{abstract}
We study the inverse problem of recovering a spatially dependent variable order in a time-fractional diffusion model from the boundary flux measurement generated by a single boundary excitation. It arises in the identification of heterogeneous media in anomalous diffusion processes. In this work,
we establish several new uniqueness results for the inverse problem in the case of piecewise constant variable orders, without any monotonicity condition. The analysis follows a new approach that combines properties of harmonic functions, a linearization technique in the Laplace domain, and tools from complex, asymptotic, and geometrical analysis. In addition, we weaken the regularity assumptions on the problem data and extend the analysis of previous contributions to higher-dimensional settings. 
\end{abstract}
\section{Introduction}
In this work, we investigate an inverse problem arising in spatially-variable order subdiffusion.
Let $\Om\subset \mathbb{R}^d$ ($d\ge2$) be an open bounded and connected domain with a Lipschitz  boundary $\partial\Omega$. Let
$\alpha:\Omega \to(0,1)$ be a spatially variable order function. Consider the following variable-order subdiffusion model for the function $U(t,x)$:
\begin{equation}\label{eq:2D3D:ibvp}
\left\{
\begin{aligned}
\partial_t^{\alpha(x)}U - \Delta_x U &=0,\quad \mbox{in }(0,\infty)\times \Omega,\\
U&=g,\quad \mbox{on }(0,\infty)\times\partial\Omega,\\
U(0,\cdot)&=0,\quad \mbox{in }\Omega,
\end{aligned}\right.
\end{equation}
where $g$ is the boundary excitation. In the model \eqref{eq:2D3D:ibvp}, the spatially variable order Caputo fractional derivative $\partial_t^{\alpha(x)} U(t,x)$ in time $t$ is defined by (see, e.g., \cite[p. 92]{KilbasSrivastavaTrujillo:2006} or \cite[p. 41]{Jin:2021})
\begin{equation}
    \partial_t^{\alpha(x)}U(t,x) = \frac{1}{\Gamma(1-\alpha(x))}\int_0^t(t-s)^{-\alpha(x)}\partial_sU(s,x)\,{\rm d}s,
\end{equation}
where $\Gamma(z)=\int_0^\infty e^{-t} t^{z-1}\,{\rm d}t$, for $\Re(z)>0$, denotes Euler's Gamma function.

The model \eqref{eq:2D3D:ibvp} describes space-dependent anomalous diffusion processes in complex media in which  heterogeneous regions exhibit spatially inhomogeneous variations.
It can be derived in the framework of continuous time random walk, with a space-dependent waiting-time distribution
\cite{Orsingher:2018,ZhangLiLuo:2013}. The model has been employed for the modeling of the evolution of a composite system with two separate regions with different subdiffusion exponents \cite{ChechkinGorenflo:2005}, subdiffusion infiltration in disordered systems \cite{KorabelBarkai:2010} and structural instability of fractional diffusion in inhomogeneous media \cite{Fedotov:2012}. In these applications, the variable order
$\alpha$  provides a fundamental characterization of the class of anomalous diffusion processes, and its identification is of paramount importance for an accurate description of medium properties.

In this work, we investigate the  inverse problem of recovering the variable order $\alpha$ from over-posed boundary data generated by a single Dirichlet boundary condition $g$, and establish several results on unique  determination.  More precisely, for all $d\in\mathbb{N}\backslash\{1\}$, we investigate the following inverse problem:
\begin{itemize}
	\item[{\bfseries (IP):}] Determine the variable order $\alpha$ from the flux data $\{\p_\nu U(t,x)\}_{t\in I,\,x\in\p\Om}$,
where $I$ is an arbitrary subset of $(0,\infty)$ with at least one positive accumulation point and $U$ is the solution of problem \eqref{eq:2D3D:ibvp} for a suitably chosen Dirichlet boundary condition $g$.
\end{itemize}

The recovery of the fractional order $\alpha$ is a central issue in inverse problems for subdiffusion \cite{JinRundell:2015,KaltenbacherRundell:2023,LiLiuYamamoto:2019}. This topic has attracted substantial interest within the mathematical community. The majority of the existing literature is devoted to the identification of one or multiple constant orders  \cite{JiKi2,JK,LY,LZ} or distributed-order  \cite{JiKi,LILuY,RuZ}. In stark contrast, the analysis of \textbf{(IP)} in the case of a variable-order $\alpha$ remains scarce. This lack of results is primarily attributed to the inherent complexity of \textbf{(IP)} when the order $\alpha:\Omega\to (0,1)$ is a space-dependent function, rather than a constant parameter that may be determined from the  asymptotic behavior in time of the solution $U$ of \eqref{eq:2D3D:ibvp} as $t\to\infty$ or $t\to0$.

Several works have investigated  \textbf{(IP)} with a spatially dependent variable order $\alpha$. One of the earliest contributions in this direction is \cite{Kian:2018:TFD}, which employs infinitely many measurements, namely measurements associated with each Dirichlet excitation $g$ belonging to an infinite-dimensional space. One of the first results addressing \textbf{(IP)} can be found in \cite{IkehataKian:2023}, which gives qualitative properties of the variable order $\alpha$ by means of the enclosure method (see e.g. \cite{Ikk1}), for a restricted and non-explicit class of Dirichlet excitations $g$. More recently, the works \cite{HJK:2025:ISDVO,HJK:2026:2D3D} established several uniqueness results for \textbf{(IP)} in spatial dimensions $d\in\{1,2,3\}$ using the flux measurement at one boundary point, under a strong monotonicity assumption on admissible candidates (that is, the possible realizations of $\alpha$ belong to a totally ordered set). In this work, we substantially strengthen the contributions of \cite{Kian:2018:TFD,HJK:2025:ISDVO,HJK:2026:2D3D} by solving \textbf{(IP)}, completely removing the monotonicity condition for the admissible piecewise constant orders $\alpha$; see Section \ref{sec:main} for detailed statements and further discussions. This is achieved through the use of several novel analytical tools, including holomorphic extensions of carefully constructed auxiliary functions and refined directional asymptotic analyses, which differ markedly from those employed in the existing literature \cite{HJK:2025:ISDVO,HJK:2026:2D3D}.

The analysis relies on a novel methodology that combines several mathematical arguments. We first transform the problem into the Laplace domain by exploiting analytic properties of problem \eqref{eq:2D3D:ibvp}. Then, by employing a Dirichlet boundary excitation $g$ satisfying Assumption~\ref{ass:g:exptype} and using a linearization of the Laplace transform in time
$\widehat{U}(p,\cdot):=\int_0^\infty U(t,\cdot) e^{-pt}\,{\rm d}t$, of the solution $U$ of \eqref{eq:2D3D:ibvp}  at frequency $p=1$, we derive a key orthogonal identity in Lemma \ref{lemma:uniqueness:main:identity}:
\begin{equation}\label{Int}
    \int_\Om (\alpha^1(x)-\alpha^2(x)) e^{x\cdot(\omega+\omega_0)}\,{\rm d} x = 0,\quad\forall \omega\in\mathbb{S}^{d-1},
\end{equation}
where $\omega_0\in\mathbb{S}^{d-1}$, and $\alpha^1$ and $\alpha^2$ are the candidates for the target $\alpha$. The proof of uniqueness is completed by proving $\alpha^1=\alpha^2$ almost everywhere in the domain $\Om$. To this end, we rewrite equation \eqref{Int} using the hyperspherical coordinates to express the vector $\omega\in \mathbb{S}^{d-1}$, and define the complex variable extension of the integral, 
which is holomorphic in each variable in $\mathbb{C}$ (cf. Lemma \ref{lemma:exponential:in:integralkernel}). By the unique continuation property of holomorphic functions, the identity \eqref{Int} still holds for the complex variable extension (cf. Remark \ref{remark:extended:identity}).
Moreover, under suitable assumptions, the integrals have explicit expressions (cf. Lemmas \ref{lemma:Bessel} and \ref{lemma:simplex:integral} for real variables and Lemmas \ref{lemma:Bessel:complexversion} and \ref{lemma:simplex:integral:complex} for complex variables).
By a delicate asymptotic analysis along an appropriate half-line in the complex plane, we derive $\alpha^1=\alpha^2$ in various settings for the unique determination of $\alpha$.

The core of our analysis concerns an inverse problem for elliptic equations in the Laplace domain associated with \eqref{eq:2D3D:ibvp} (see \eqref{eq:afterLaptrans} in Section 3), whose objective is the identification of a piecewise constant variable order $\alpha$ from a single boundary measurement. We rely on various properties of elliptic equations, representation formulas as well as the study of  the key orthogonality identity \eqref{Int}, which constitute the main ingredients in the proof of our principal results. This analysis, as well as the formulation of the problem in the Laplace domain, are also connected with the classical inverse problem of identifying inhomogeneities or inclusions, a topic that has been extensively investigated in the  literature \cite{AIP,ADKL,DLLi,DLLi2,FrIs,FV,Ik98,KaSe,LCY,TrTs}; see \cite{AK} for a comprehensive review.

The rest of the paper is organized as follows. In Section \ref{sec:main}, we describe the main results and provide relevant further discussions. Then in Section \ref{sec:direct}, we collect preliminaries about the direct problem, including well-posedness, linearization of the Laplace transform of the solution, and a crucial integral identity. In Sections \ref{sec:inverse} and \ref{section:polygons}, we analyze the cases of spherical inclusions and polygonal inclusions, respectively. For any integer $n$, the notation $[n]$ denotes the set $\{1,\ldots,n\}$, and $\mathbb{R}_+=(0,\infty)$.

\section{Main results and discussions}\label{sec:main}

In this section, we state the main theoretical findings, i.e., the uniqueness of the inverse problem of determining the spatially varying order $\alpha$ in problem \eqref{eq:2D3D:ibvp} from the boundary flux data $\p_\nu U|_{I\times\p\Om}$. This requires suitable conditions on the variable order $\alpha$, the Dirichlet boundary excitation $g$ and the set $I$.

\begin{ass}
\label{ass:alpha:general}
    $\alpha:\Om\to(0,1)$ is measurable and satisfies $\mathrm{esssup}_{x\in \Omega}\alpha(x) <2\mathrm{essinf}_{x\in\Omega}\alpha(x)$.
\end{ass}
\begin{ass}
\label{ass:g:exptype}
    There exist some $k\in\mathbb{N}\backslash\{1\}$ and $\omega_0\in\mathbb{S}^{d-1}$ such that
    $$g(t,x)=t^k e^{x\cdot\omega_0},\quad\forall (t,x)\in(0,\infty)\times\p\Om.$$
$I$ is a subset of $(0,\infty)$ with at least one positive accumulation point.
\end{ass}
We prove in Section \ref{sec:direct} that, under Assumptions \ref{ass:alpha:general} and \ref{ass:g:exptype} on $\alpha$ and $g$, problem  \eqref{eq:2D3D:ibvp} admits a unique solution $U\in C^1([0,\infty);L^2(\Omega))\cap C([0,\infty);H^1(\Omega))$ with $\p_\nu U\in C([0,\infty);H^{-\frac 1 2}(\p\Omega))$. We will study \textbf{(IP)} for a piecewise constant variable order $\alpha$.
\subsection{Spherical inclusions}
For any $r>0$ and $x\in\mathbb{R}^d$, let $B_r(x)$ be a ball of radius $r$ and center $x$ in $\mathbb{R}^d$.
Also, fix $\alpha_{\rm in}$ to be any function satisfying Assumption \ref{ass:alpha:general}. For a set $A\subset \mathbb R^d$, $\mathds{1}_A$ denotes the characteristic function of $A$.
\begin{ass}\label{ass:alpha}
Let $\alpha$ satisfy Assumption \ref{ass:alpha:general}. Also, suppose that there exist some $N\in\mathbb{N}$, $\{r_j\}_{j=1}^N\subset(0,\infty)$, $\{x_j\}_{j=1}^N\subset\Om$ and $\{\alpha_j\}_{j=1}^N\subset(-1,1)$ such that $\{\overline{B_{r_j}(x_j)}\}_{j=1}^N$ is a family of $($possibly intersecting$)$ balls satisfying $\overline{B_{r_j}(x_j)}\subset\Om$ for all $1\le j\le N$ and
\begin{equation}\label{al}
    \alpha(x)=\alpha_{\rm in}(x)+\sum_{j=1}^N \alpha_j \mathds{1}_{B_{r_j}(x_j)}(x),\quad\forall x\in\Om.
\end{equation}

\end{ass}

	\begin{figure}[h]
		\centering
		\begin{tabular}{ccc}
        \includegraphics[width=0.23\linewidth]{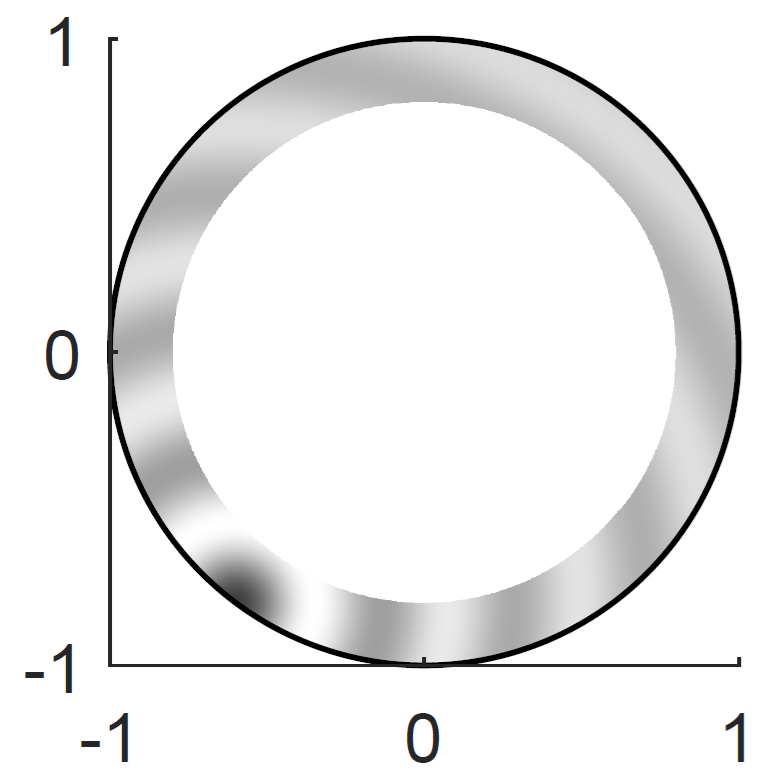}
		&\includegraphics[width=0.23\linewidth]{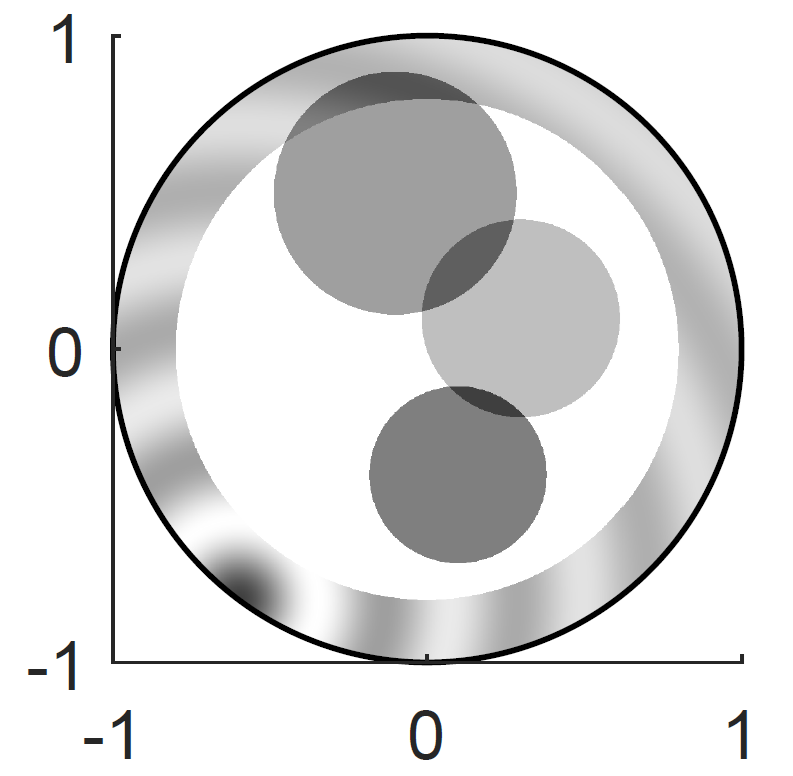}
		&\includegraphics[width=0.23\linewidth]{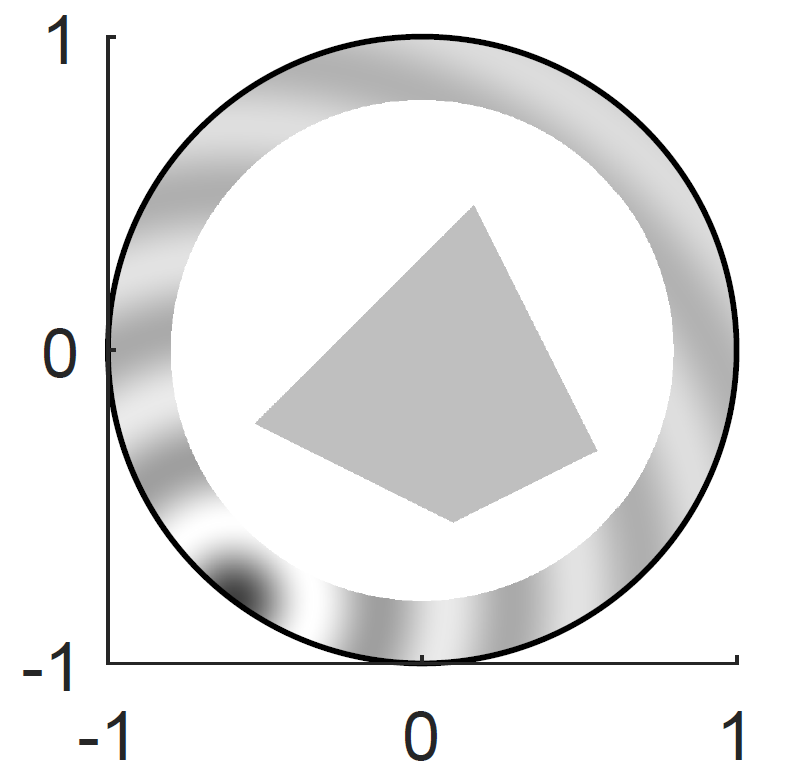}
		\includegraphics[width=0.065\linewidth]{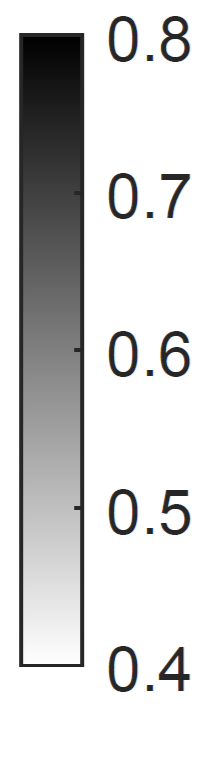}\\
        (a) & (b) & (c)
    \end{tabular}
		\caption{Examples of $\alpha$ defined on $\Om=B_1(0)$ with (a) $\alpha_{\rm in}$, (b) and (c)  candidate $\alpha$ profiles satisfying the requirements of Theorems \ref{theorem:balls} and \ref{theorem:2D:convex}, respectively. }\label{fig:Thms:21:23}
	\end{figure}	

The next result gives the unique identifiability of spherical inclusions and their amplitudes. Fig. \ref{fig:Thms:21:23}(b) shows one $\alpha$ (with $\alpha_{\rm in}$ in (a)) satisfying the requirement of  Theorem \ref{theorem:balls} (i.e., Assumption \ref{ass:alpha}).
\begin{theorem}
\label{theorem:balls}
Let Assumption \ref{ass:g:exptype} hold. Let  $U=U^i$ be the solution to  \eqref{eq:2D3D:ibvp} with $\alpha=\alpha^i$ for $i=1,2$ satisfying Assumption \ref{ass:alpha}.
If $\p_\nu U^1(t,x)=\p_\nu U^2(t,x)$ for all $(t,x)\in I\times\p\Om$, then we have $\alpha^1=\alpha^2$.
\end{theorem}

To the best of our knowledge, Theorem~\ref{theorem:balls} represents the first resolution to \textbf{(IP)} that does not rely on any monotonicity assumption. Specifically, Theorem~\ref{theorem:balls} applies to any variable order $\alpha$ of the form~\eqref{al} and, in contrast to  existing studies~\cite{HJK:2025:ISDVO,HJK:2026:2D3D}, does not require any monotonicity condition. This relaxation is of great importance, both in view of potential practical applications and from a purely mathematical standpoint, as the recovery of a piecewise constant variable order $\alpha$ is a highly nonlinear and intrinsically challenging inverse problem. Moreover, the unique Dirichlet boundary excitation $g$ is prescribed in an explicit form in Assumption~\ref{ass:g:exptype}, which can potentially be useful for numerical reconstruction.

The proof of Theorem~\ref{theorem:balls} is based on a novel methodology that combines two fundamental ingredients, which appear to be new in the analysis of \textbf{(IP)}. The first ingredient consists in the linearization of the data in the Laplace domain, which leads to the key orthogonality identity~\eqref{Int}. The second ingredient relies on the derivation of the uniqueness result through analytic continuation of the identity~\eqref{Int} into the complex domain, together with the use of delicate directional asymptotic analyses. The most delicate part of the proof of Theorem~\ref{theorem:balls} lies in the simultaneous identification of the number $N$ of discontinuity interfaces, and the centers $\{x_j\}_{j=1}^N$ and radii $\{r_j\}_{j=1}^N$ of the discontinuity balls $\{B_{r_j}(x_j)\}_{j=1}^N$. Importantly, all these parameters are determined without imposing any additional assumptions.

One of the central steps in the proof of Theorem~\ref{theorem:balls} consists in reformulating \textbf{(IP)} as a family of elliptic boundary value problems in the Laplace domain. Hence \textbf{(IP)} is closely connected to the inverse problem of recovering a piecewise constant potential in an elliptic equation from a single boundary measurement. Note that, in the absence of any monotonicity type of assumption, the resolution of the latter problem remains open when $d\geq3$ \cite{Su1,Su2}.

Besides relaxing the monotonicity assumption imposed in the works \cite{HJK:2025:ISDVO,HJK:2026:2D3D} for \textbf{(IP)}, the present work also significantly weakens the regularity requirement on the underlying domain $\Omega$: it may possess merely Lipschitz regularity, whereas previous works \cite{Kian:2018:TFD,HJK:2025:ISDVO,HJK:2026:2D3D} typically assume $C^{1,1}$ regularity. This relaxation is made possible by a careful adaptation of the linearization step in the analysis to the more singular setting, in which the $H^2(\Omega)$ regularity of solutions to the associated elliptic boundary value problems is no longer guaranteed.

\subsection{Polygonal inclusions}
Let $\{\mathrm{e}_i\}_{i=1}^d$ be an orthonormal basis of $\mathbb{R}^d$.
Let $T\subset\mathbb{R}^d$ be the $d$-simplex with the set of vertices being $\{\mathbf{0}\}\cup\{\mathrm{e}_i\}_{i=1}^d$:
\begin{equation}\label{eq:T:definition}
    T=\left\{\sum_{i=1}^d x_i\mathrm{e}_i\,:\,\sum_{i=1}^d x_i<1\mbox{ and }x_j>0\mbox{ for } j=1,\dots,d\right\}.
\end{equation}
For any square matrix $V\in\mathbb{R}^{d\times d}$  with $\det V>0$ and $x\in\mathbb{R}^d$, we denote by $T_{x}(V):=x+V T$ the $d$-simplex with the set of vertices being $\{x\}\cup\{x+V\mathrm{e}_i\}_{i=1}^d$. Note that any nondegenerate simplex can be expressed as $T_{x}(V)$ for some square matrix $V\in\mathbb{R}^{d\times d}$ with $\det V>0$ and  $x\in\mathbb{R}^d$.

The following assumption facilitates the proof of uniqueness. The term ``irreducible" in the statement ``an irreducible set of simplices $\{T_{x_i}(V_i)\}_{i=1}^N$" means that the statement is no longer true if one excludes any simplex in the set $\{T_{x_i}(V_i)\}_{i=1}^N$.
\begin{ass}\label{ass:alpha:simplices}
Let Assumption \ref{ass:alpha:general} be satisfied for both $\alpha=\alpha^1$ and $\alpha=\alpha^2$. Also, suppose that there exists a finite set of disjoint simplices $\{T_{x_i}(V_i)\}_{i=1}^N$ such that $\alpha^1-\alpha^2$ is constant in $T_{x_i}(V_i)$ for each $i\in [N]$ and is zero in $\Om\backslash\bigcup_{i=1}^N \overline{T_{x_i}(V_i)}$. Moreover, if $\alpha^1\not \equiv\alpha^2$, for the convex hull $P$ of the union of an irreducible set of simplices $\{T_{x_i}(V_i)\}_{i=1}^N$, there exist some $i\in[N]$ and a vertex of $T_{x_i}(V_i)$ that is also a vertex of $P$ but is not a vertex of any simplex in $\{T_{x_k}(V_k)\}_{k\ne i}$.
\end{ass}

\begin{figure}[hbt!]
\centering\setlength{\tabcolsep}{0pt}
\begin{tabular}{cccc}
\includegraphics[width=0.24\linewidth]{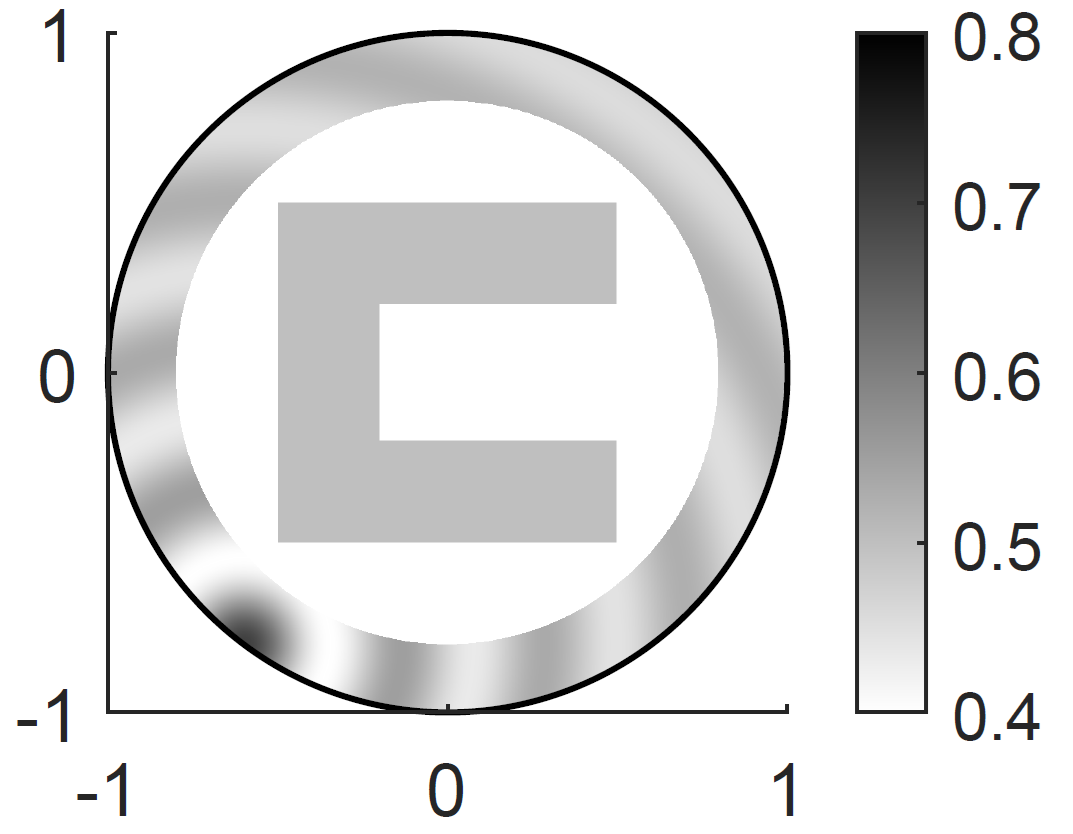}&\includegraphics[width=0.24\linewidth]{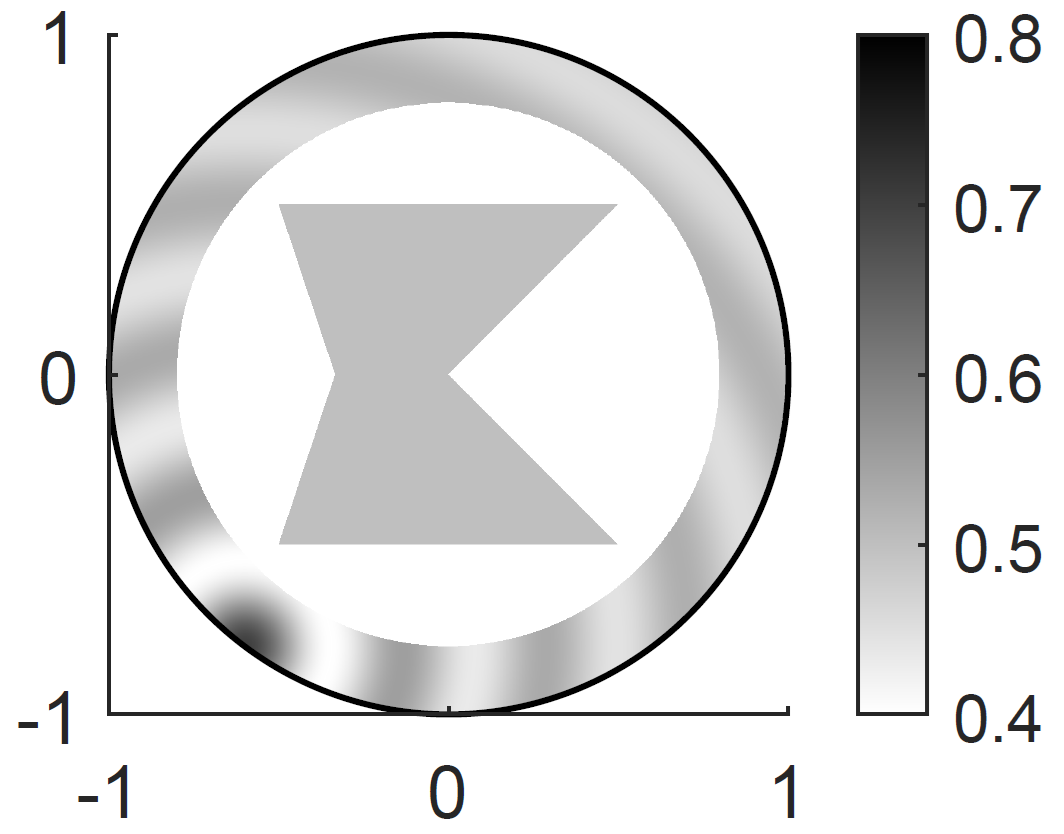}&\includegraphics[width=0.24\linewidth]{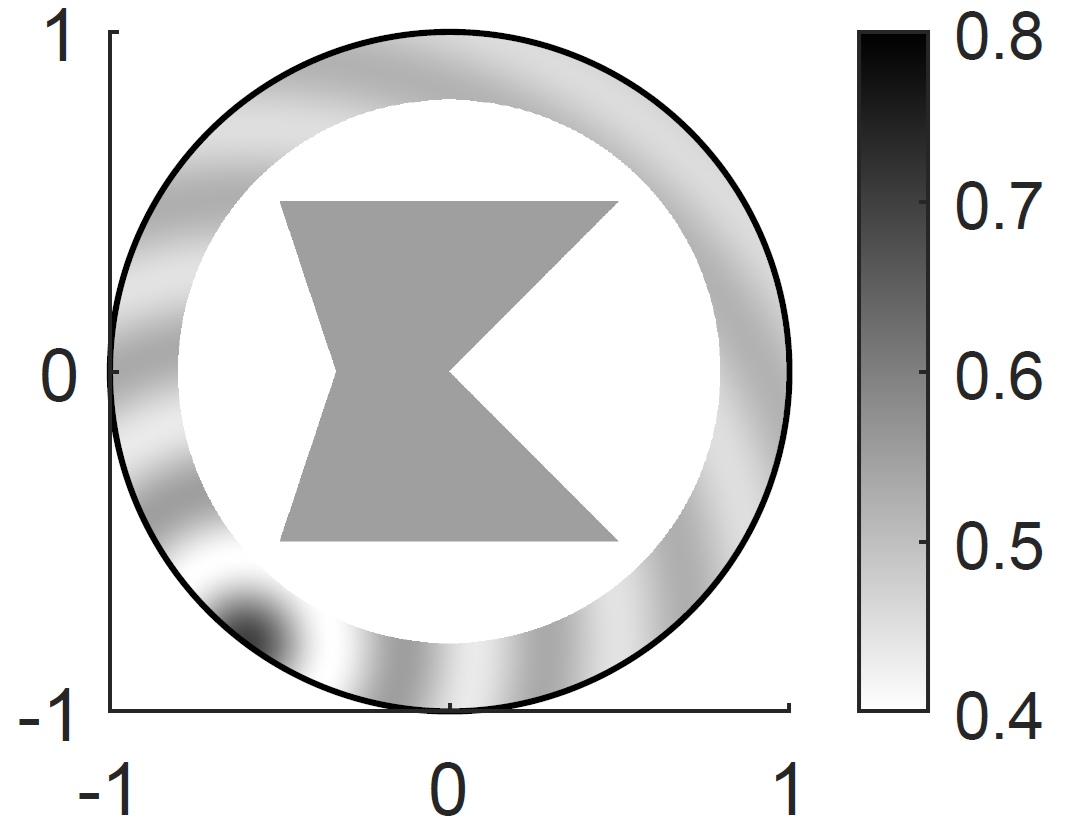}\\
$\alpha^1$ & $\alpha^2$ & $\alpha^3$
\end{tabular}
\caption{\label{fig:examples}
A schematic illustration of Assumption \ref{ass:alpha:simplices} on the variable order $\alpha$.
}
\end{figure}

\begin{figure}[hbt!]
\centering\setlength{\tabcolsep}{0pt}
\begin{tabular}{cccc}
\includegraphics[width=0.24\linewidth]{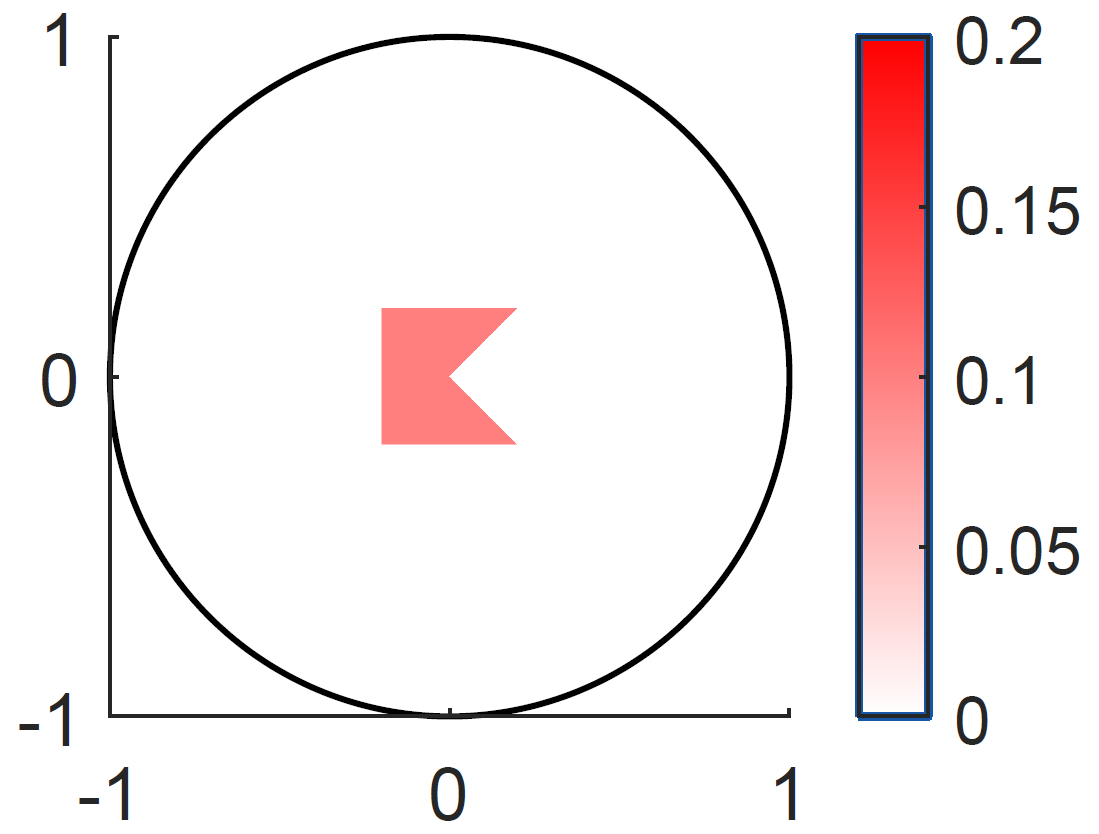}&
\includegraphics[width=0.24\linewidth]{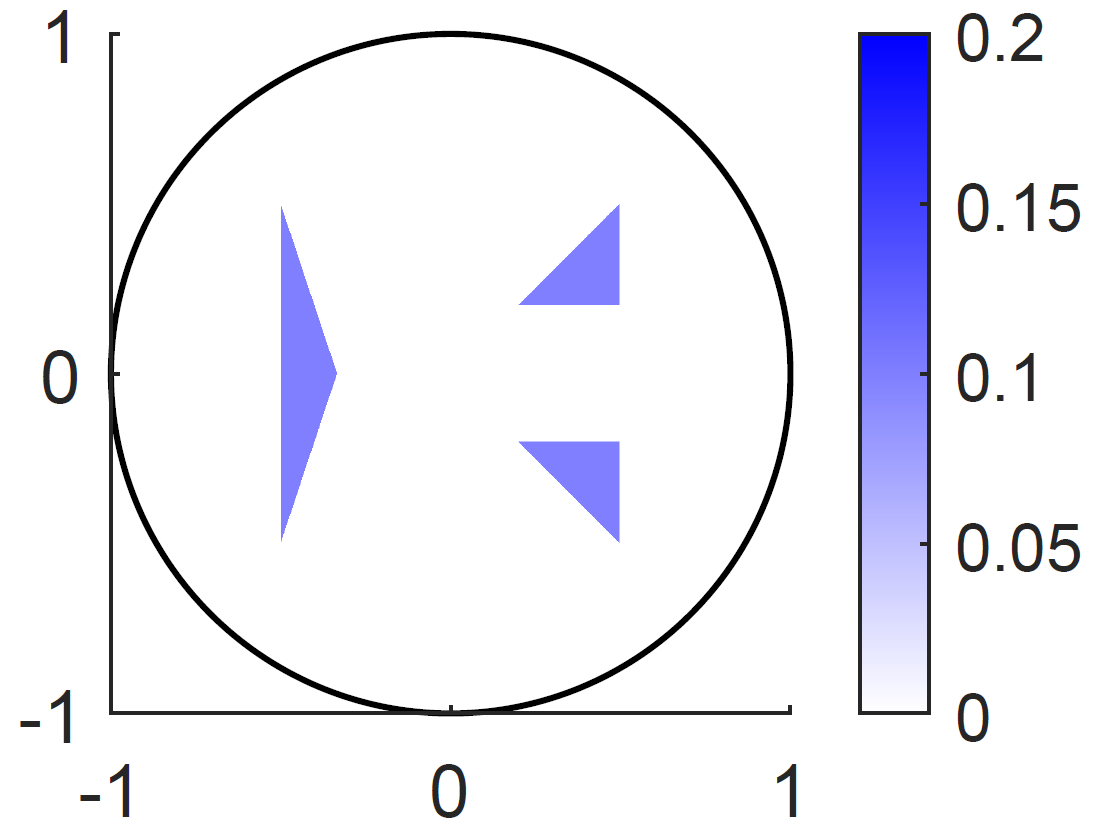}&
\includegraphics[width=0.18\linewidth]{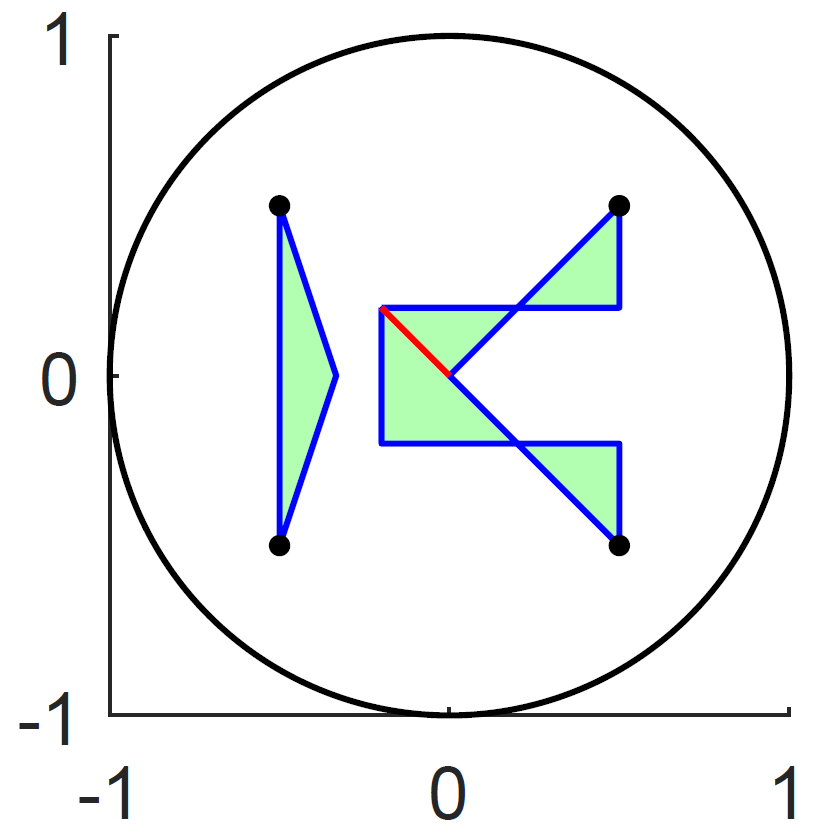}\\
$\max\left(0,\alpha^2-\alpha^1\right)$ & $\max\left(0,\alpha^1-\alpha^2\right)$ & 
$\operatorname{supp}(\alpha^1-\alpha^2)$
\\[5mm]
\includegraphics[width=0.24\linewidth]{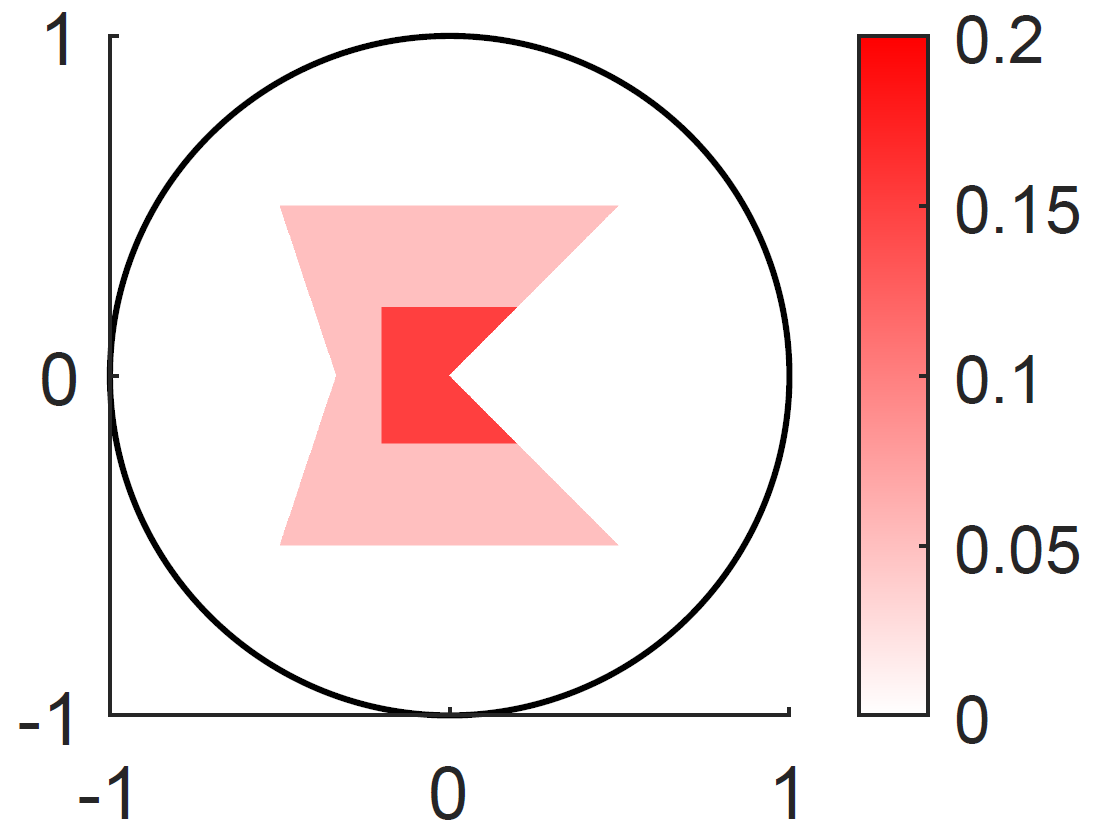}&
\includegraphics[width=0.24\linewidth]{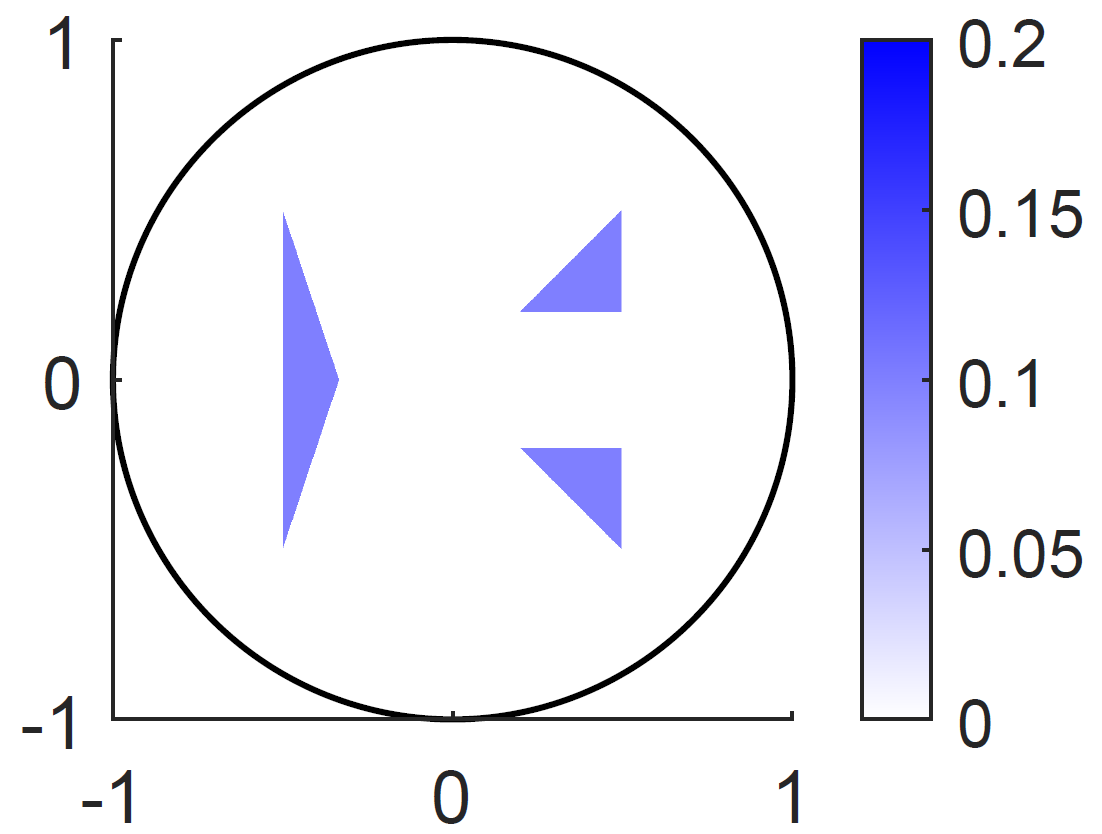}&
\includegraphics[width=0.18\linewidth]{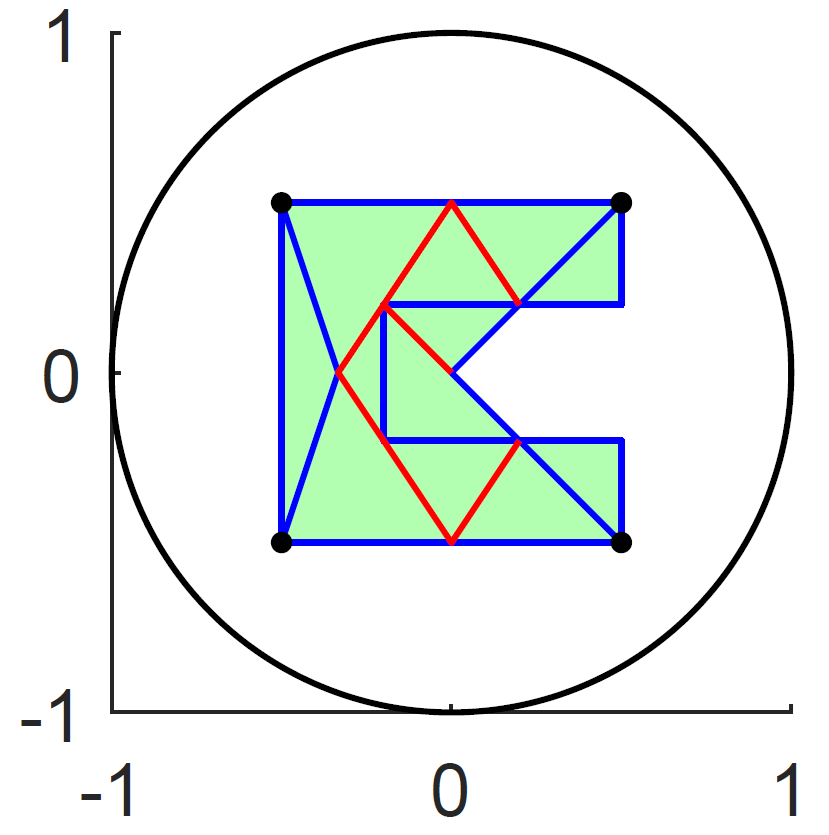}\\
$\max\left(0,\alpha^3-\alpha^1\right)$ & $\max\left(0,\alpha^1-\alpha^3\right)$ & 
$\operatorname{supp}(\alpha^1-\alpha^3)$\\
(a) & (b) & (c)
\end{tabular}
\caption{\label{fig:Assumption:violated}
The differences $\alpha^1-\alpha^i$  ($i=2,3$) in Fig. \ref{fig:examples}.
The pair $(\alpha^1,\alpha^2)$ satisfies Assumption \ref{ass:alpha:simplices}, whereas the pair $(\alpha^1,\alpha^3)$ violates Assumption \ref{ass:alpha:simplices}.
In (c), the regions $\operatorname{supp}(\alpha^1-\alpha^2)$ and $\operatorname{supp}(\alpha^1-\alpha^3)$ are shaded in green, and the blue and red line segments indicate the necessary and optional choices for the decomposition by nondegenerate simplices, respectively.
}
\end{figure}

In Figs. \ref{fig:examples} and \ref{fig:Assumption:violated}, we present two pairs of $\alpha$ satisfying Assumption \ref{ass:alpha:general}, one satisfying Assumption \ref{ass:alpha:simplices} and the other violating Assumption \ref{ass:alpha:simplices}.
Fig. \ref{fig:examples} shows three cases $\alpha^i$ ($i=1,2,3$) of $\alpha$ that satisfies Assumption \ref{ass:alpha:general}, and Fig. \ref{fig:Assumption:violated} illustrates their differences from several perspectives.
In Fig. \ref{fig:Assumption:violated}(c), the blue line segments are the discontinuity interfaces of $\alpha^1-\alpha^2$ (top) and $\alpha^1-\alpha^3$ (bottom), whereas the red line segments are auxiliary for constructing the simplicial decomposition.
Note that the convex hull $P$ of $\operatorname{supp}(\alpha^1-\alpha^2)$ and that of $\operatorname{supp}(\alpha^1-\alpha^2)$ are both a square.
In Fig. \ref{fig:Assumption:violated} (c) (top), the simplices shaded in green with boundaries being blue and red line segments satisfy Assumption \ref{ass:alpha:simplices} for the pair $(\alpha^1,\alpha^2)$.
However, Fig. \ref{fig:Assumption:violated}(c) (bottom) shows that Assumption \ref{ass:alpha:simplices} does not hold for the pair $(\alpha^1,\alpha^3)$, since every vertex of the convex hull $P$ of $\operatorname{supp}(\alpha^1-\alpha^3)$ is an end point of three blue line segments that are the interfaces of discontinuities of $\alpha^1-\alpha^3$.

The next result gives the unique determination of polygonal inclusions and their amplitudes.
\begin{theorem}
\label{theorem:simplex}
Let Assumption \ref{ass:g:exptype} hold.
Let $U=U^i$ be the solution to \eqref{eq:2D3D:ibvp} with $\alpha=\alpha^i$ for $i=1,2$ satisfying Assumption \ref{ass:alpha:simplices}.
If $\p_\nu U^1(t,x)=\p_\nu U^2(t,x)$ for all $(t,x)\in I\times\p\Om$,
then we have $\alpha^1=\alpha^2$.
\end{theorem}

\begin{definition}\label{definition:convexpolygon}
Let $\{x_i\}_{i=1}^n\subset\mathbb{R}^d$ be a finite set of points  such that there is no hyperplane in $\mathbb{R}^d$ that contains all the elements $\{x_i\}_{i=1}^n$. Let
    \begin{equation}\label{eq:convex:polytope}
        \mathcal{C}:=\left\{\sum_{i=1}^n c_ix_i\,:\,\sum_{i=1}^n c_i=1\mbox{ and }c_i\ge0\,\forall i\in[n]\right\}.
    \end{equation}
   Then the interior of the set $\mathcal{C}$ is said to be a convex polygon.
    If the set $\{x_i\}_{i=1}^n$ is irreducible in \eqref{eq:convex:polytope}, every $x_i$ is said to be a vertex of the convex polygon.
\end{definition}

The next result on the unique determination of a convex polygon in $\mathbb{R}^2$ is direct from Theorem \ref{theorem:simplex}. Fig. \ref{fig:Thms:21:23} (c) shows a variable order $\alpha$ satisfying the requirement of Theorem \ref{theorem:2D:convex}.
\begin{theorem}\label{theorem:2D:convex}
Let $d=2$, and let Assumption \ref{ass:alpha:general} with $\alpha=\alpha_{\rm in}$ hold. For each $i\in\{1,2\}$, let $\alpha_{\rm c}^i$ be a constant, $D^i$ a convex polygon such that $\overline{D^i}\subset\Om$, and  $\alpha^i(x) := \alpha_{\rm in}(x) + \alpha_{\rm c}^i\mathds{1}_{D^i}(x)$.
Suppose that both $\alpha=\alpha^1$ and $\alpha=\alpha^2$ satisfy Assumption \ref{ass:alpha:general}.
Let $U=U^i$ be the solution to problem \eqref{eq:2D3D:ibvp} with $\alpha=\alpha^i$ for $i=1,2$.
If $\p_\nu U^1(t,x)=\p_\nu U^2(t,x)$ for all $(t,x)\in I\times\p\Om$,
then we have $\alpha_{\rm c}^1=\alpha_{\rm c}^2$ and $D^1=D^2$.
\end{theorem}

In Theorems~\ref{theorem:simplex} and~\ref{theorem:2D:convex}, we have extended the analysis of Theorem~\ref{theorem:balls} to the case of polygonal inclusions. This extension is achieved under the technical conditions in Assumption~\ref{ass:alpha:simplices}. While the assumption imposes certain restrictions on the class of polygons in Theorem~\ref{theorem:simplex} when $d \geq 3$, we prove in Theorem~\ref{theorem:2D:convex} that, for $d = 2$, it is satisfied for a very general class of polygonal inclusions. The extension of these results to the recovery of multiple polygonal inclusions is discussed in Section~\ref{subsection:various:cases}.
These results not only cover the inclusions beyond the ball-shaped ones (in Theorem~\ref{theorem:balls}), but also confirm the potential of extending the analysis to other classes of inclusions, which we leave for future investigation.

Like Theorem~\ref{theorem:balls}, the proofs of Theorems~\ref{theorem:simplex} and~\ref{theorem:2D:convex} are based on the orthogonality relation~\eqref{Int}, and tools from complex analysis and directional asymptotic properties. However, in contrast to Theorem~\ref{theorem:balls}, the treatment of polygonal inclusions requires additional geometrical analysis due to the more intricate structure of polygonal inclusions.
The proof of Theorem \ref{theorem:2D:convex} is given in  Section \ref{subsect:2D:convex}. Further several other implications of Theorem \ref{theorem:simplex} are given in Sections \ref{section:rectangles} and \ref{subsection:various:cases}, which are obtained by proving that Assumption \ref{ass:alpha:simplices} is a necessary condition for various classes of candidates for $\alpha$.
Note also that we still have a uniqueness result without Assumption \ref{ass:alpha:simplices} but with a different constraint (cf. Theorem \ref{theorem:cones:manysides}).

\section{Preliminary results}\label{sec:direct}
In this section, we develop key analytic properties of problem \eqref{eq:2D3D:ibvp}, which are crucial to the analysis of the inverse problem in Sections \ref{sec:inverse} and \ref{section:polygons}.
We define the Hilbert space $H_\Delta(\Omega):=\{v\in H^1(\Omega):\ \Delta v\in L^2(\Omega)\}$ equipped with the norm
$$\|v\|_{H_\Delta(\Omega)}^2=\|v\|_{H^1(\Omega)}^2+\|\Delta v\|_{L^2(\Omega)}^2,\quad v\in H_{\Delta}(\Omega).$$
Note also that, following \cite[Lemma 2.2]{Kav}, 
for $v\in H_\Delta(\Omega)$, the normal derivative  $\partial_\nu v|_{\partial\Omega}$ of $v$ is well defined as an element of $H^{-\frac{1}{2}}(\partial\Omega)$. Moreover, there exists $C>0$ depending only on $\Om$ such that
\begin{equation}\label{normal}\|\partial_\nu v\|_{H^{-\frac{1}{2}}(\partial\Omega)}\leq C\|v\|_{H_\Delta(\Omega)},\quad v\in H_\Delta(\Omega).\end{equation}

Then we can state the following unique existence result on the forward problem.
\begin{theorem}
\label{theorem:analyticity}
Under Assumptions \ref{ass:alpha:general} and  \ref{ass:g:exptype}, problem \eqref{eq:2D3D:ibvp} admits a unique solution\\  $U\in C^1([0,\infty);L^2(\Omega))\cap C([0,\infty);H_\Delta(\Omega))$. Moreover, the map $t\mapsto U(t,\cdot)$ is analytic on $\mathbb{R}_+$
as a map taking values in $H_\Delta(\Omega)$, $t\mapsto U(t,\cdot)e^{-pt} \in L^1(0,\infty;H_\Delta(\Omega))$ and the Laplace transform in time $\widehat{U}(p,\cdot):=\int_0^\infty U(t,\cdot)e^{-pt}\d t$ of $U$ belongs to $H_\Delta(\Omega)$ for all $p>0$.
\end{theorem}
\begin{proof}
Let $A$ be the unbounded operator with its domain $D(A):=\{v\in H^1_0(\Omega):\ \Delta v\in L^2(\Omega)\}$, defined by $Av=-\Delta v$ for $v\in D(A)$. In view of \cite[Proposition 2.1]{Kian:2018:TFD},  for all $p\in\mathbb C\setminus(-\infty,0]$, the map $A+p^\alpha$ is boundedly invertible. Moreover, by \cite[Section 6.2]{Ev},  $D(A)$ embeds continuously into $H_\Delta(\Omega)$ and, for all $v\in L^2(\Omega)$ and $p\in\mathbb C\setminus(-\infty,0]$, we obtain
\begin{equation}\label{esti}\|(A+p^\alpha)^{-1}v\|_{H_\Delta(\Omega)}\leq C(\|A(A+p^\alpha)^{-1}v\|_{L^2(\Omega)}+\|(A+p^\alpha)^{-1}v\|_{L^2(\Omega)}),\end{equation}
with $C>0$ depending only on $\Omega$.
Then the unique existence of solutions can be deduced by combining the arguments of \cite[Proposition 3.1]{Kian:2018:TFD} with \cite[p. 16]{Ki1},  similarly to \cite[Theorem 3.1]{HJK:2025:ISDVO}. The time analyticity is direct from \cite[Lemma 3.2]{Kian:2018:TFD}. Note that in contrast to \cite[Proposition 3.1]{Kian:2018:TFD}, due to the weaker Lipschitz regularity of $\Omega$,  we replace the space $H^2(\Omega)$ by $H_\Delta(\Omega)$ in the proof and  applies the estimate \eqref{esti}.
\end{proof}

By combining Theorem  \ref{theorem:analyticity} with \eqref{normal}, we deduce that $\p_\nu U\in C([0,\infty);H^{-\frac 1 2}(\p\Omega))$ and $t\mapsto\p_\nu U(t,\cdot)|_{\partial\Om}$ is analytic on $\mathbb{R}_+$
as a map taking values in $H^{-\frac 1 2}(\p\Omega)$.
By \cite[Proposition 3.1]{Kian:2018:TFD}, for all $p>0$, the Laplace transform $\widehat{U}(p,\cdot)$ in time of $U(t,\cdot)$ is the unique solution of the following boundary value problem
\begin{equation}\label{eq:afterLaptrans}
\left\{\begin{aligned}
- \Delta_x \widehat{U}(p,\cdot) + p^{\alpha(x)}\widehat{U}(p,\cdot)&=0,\quad\mbox{in } \Omega,\\
\widehat{U}(p,\cdot)&=\widehat{g}(p,\cdot),\quad\mbox{on }\partial\Omega.
\end{aligned}\right.
\end{equation}
Actually the solution $U$ of problem \eqref{eq:2D3D:ibvp} is constructed from the inversion of the Laplace transform $\widehat{U}(p,\cdot)$ of the solution  of \eqref{eq:afterLaptrans}. Moreover, if $g$ satisfies Assumption \ref{ass:g:exptype}, then
\begin{equation*}
x\mapsto\widehat{g}(p,x)= k!p^{-k-1} e^{x\cdot\omega_0}\in H^{\frac 1 2}(\partial\Omega),\quad\forall p>0.
\end{equation*}

We also use the following two elementary results.
\begin{lemma}\label{l1} Let Assumption \ref{ass:alpha:general} hold. Then there exists $C>0$ depending on $\Omega$ and $\alpha$ such that
\begin{align}\label{l1a}
\|p^\alpha -1\|_{L^\infty(\Omega)}&\leq C|p-1|,\quad p\in[\tfrac12,\tfrac32],\\
\label{l1b}
\left\|(p-1)^{-1}(p^\alpha -1)-\alpha\right\|_{L^\infty(\Omega)}&\leq C|p-1|,\quad p\in[\tfrac12,\tfrac32].
\end{align}
\end{lemma}
\begin{proof}
   For almost every $x\in\Omega$ and all $p\in[\frac12,\frac32]$, applying Taylor's formula to the map $\tau\mapsto\tau^{\alpha(x)}$ gives
 \begin{equation}\label{l1c}p^{\alpha(x)}=1+\alpha(x)(p-1)+(p-1)^2\int_0^1\frac{(1-s)^2\alpha(x)(\alpha(x)-1)}{2}((1-s)+sp)^{\alpha(x)-2}\d s.\end{equation}
Meanwhile, for  all $p\in[\frac12,\frac32]$, $((1-s)+sp)^{\alpha-2} \leq 4$, and thus we get
$$\begin{aligned}&\left\|\int_0^1\frac{(1-s)^2\alpha(\alpha-1)}{2}((1-s)+sp)^{\alpha-2}\d s\right\|_{L^\infty(\Omega)}\leq \|\alpha\|_{L^\infty(\Omega)}(1+\|\alpha\|_{L^\infty(\Omega)}).\end{aligned}$$
By combining this with the identity \eqref{l1c}, we easily deduce \eqref{l1a}-\eqref{l1b}.
  \end{proof}
\begin{lemma}\label{l2} Let the assumptions of Theorem \ref{theorem:analyticity} hold. Then, we have
\begin{equation}\label{l2a}
\sup_{p\in[\frac12,\frac32]}\|\widehat{U}(p,\cdot)\|_{H^1(\Om)}<\infty.
\end{equation}

\end{lemma}
\begin{proof}
Fix $p\in[\frac12,\frac32]$. We split $\widehat{U}(p,\cdot)$ into
$$\widehat{U}(p,x)=k!p^{-k-1} e^{x\cdot\omega_0} +W_p(x),$$
where $W_p\in H_\Delta(\Omega)$ solves \begin{equation}\label{l2b}-\Delta W_p(x)+p^{\alpha(x)} W_p(x)=-k!p^{-k-1}(p^\alpha-1) e^{x\cdot\omega_0},\quad \mbox{in }\Omega.\end{equation}
Clearly,
\begin{equation}\label{l2c}\|k!p^{-k-1} e^{x\cdot\omega_0}\|_{H^1(\Omega)}\leq k!2^{k+1}\| e^{x\cdot\omega_0}\|_{H^1(\Omega)},\quad p\in [\tfrac12,\tfrac32].\end{equation}
Moreover, multiplying \eqref{l2b} by $W_p$, integrating over the domain $\Omega$ and integrating by parts yield
$$\begin{aligned}\int_\Om |\nabla W_p|^2\d x&\leq \int_\Om (|\nabla W_p|^2+p^\alpha |W_p|^2)\d x= \int_\Om (-\Delta W_p+p^\alpha W_p)W_p\d x\\
&= -\int_\Om k!p^{-k-1}(p^\alpha-1) e^{x\cdot\omega_0}W_p\d x\leq 2^{k}k! \| e^{x\cdot\omega_0}\|_{L^2(\Omega)}\| W_p\|_{L^2(\Omega)}.\end{aligned}$$
Thus, by Poincar\'e inequality, there exists $C>0$ depending only on $\Omega$ such that
$$\|W_p\|_{H^1(\Omega)}^2\leq C2^{k}k!\| e^{x\cdot\omega_0}\|_{L^2(\Omega)}\| W_p\|_{H^1(\Omega)}.$$
Therefore, we obtain
$$\|W_p\|_{H^1(\Omega)}\leq C2^{k}k!\| e^{x\cdot\omega_0}\|_{L^2(\Omega)},\quad p\in [\tfrac12,\tfrac32].$$
By combining this with \eqref{l2c}, we obtain \eqref{l2a}.
\end{proof}

The following linearization result plays a central role in the analysis.
\begin{definition}
\label{def:gen:derivative}
    Let $X$ be a real normed vector space and $O$ be any nonempty open subspace of $\mathbb{R}$. A map $f : O \to X$ is said to be differentiable at $x \in O$ if there exists an element $f'(x)$ of $X$, called the {\rm(}generalized{\rm)} derivative of $f$ at $x$, such that
    $$\lim_{\varepsilon\to 0}\frac{\|f(x+\varepsilon) -f(x)-\varepsilon f'(x)\|_{X}}{\varepsilon}=0.$$
\end{definition}

\begin{theorem}[Linearization]
\label{theorem:flux:freqisone:asympt:general}
Let $\alpha$ and $g$ satisfy Assumptions \ref{ass:alpha:general} and \ref{ass:g:exptype}, respectively.
Let $U$ be the solution of problem \eqref{eq:2D3D:ibvp}.
Then the map $G:(0,\infty)\to H^{-\frac{1}{2}}(\partial\Omega)$ defined by
$$G(p)(y):=\p_\nu\widehat{U}(p,y),\quad\forall p\in(0,\infty),\,y\in\p\Om$$
is differentiable at $p=1$ $($in the sense of Definition \ref{def:gen:derivative} with $X=H^{-\frac{1}{2}}(\partial\Omega)$$)$ and the derivative $G'(1)$ is given by
\begin{equation}\label{eq:flux:lowfreq:asympt:pnear1}
G'(1)=-(k+1)! \p_\nu v_{0} - k! \p_\nu v_1,
\end{equation}
where $v_{0},v_1\in H_\Delta(\Om)$ are the unique solutions of
\begin{equation}\label{eq:bvp:v0}
\left\{
\begin{aligned}
-\Delta v_{0} + v_{0}&=0,\quad\mbox{in }\Om,\\
v_{0}&=e^{x\cdot\omega_0},\quad\mbox{on }\partial\Omega
\end{aligned}
\right.
\quad\mbox{and}\quad\left\{
\begin{aligned}
-\Delta v_{1}+v_{1}&= \alpha v_{0},\quad\mbox{in }\Om,\\
v_{1}&=0,\quad\mbox{on }\partial\Omega.
\end{aligned}
\right.
\end{equation}
\end{theorem}
\begin{proof}
For any $p>0 $ and $ x\in\Om$, let
$$V_p(x):=\widehat{U}(p,x)-\widehat{U}(1,x) + (p-1)\left[(k+1)! v_0(x) + k! v_1(x)\right].$$
Then, for each $p>0$,  $V_p\in H_\Delta(\Omega)$ is the unique weak solution to
\begin{equation}\label{eq:V:problem}
\left\{
\begin{aligned}
(-\Delta + 1)V_p &=f_p,\quad \mbox{in }\Om,\\
V_p &=g_p,\quad \mbox{on }\partial\Omega,
\end{aligned}
\right.
\end{equation}
with
$f_p=(1-p^{\alpha(x)})\widehat{U}(p,\cdot) + (p-1)k!\alpha v_0$ and $g_p=k!e^{x\cdot\omega_0}(p^{-k-1}-1) + (p-1)(k+1)! e^{x\cdot\omega_0}.$
By classical properties of elliptic equations  \cite[Section 6.2]{Ev}, there exists  $C>0$ depending only on $\Omega$ such that
$$\|V_p\|_{H_\Delta(\Om)}\leq C(\|f_p\|_{L^2(\Om)}+\|g_p\|_{H^{\frac 1 2}(\p\Om)}),\quad p>0.$$
For all $p\in[\frac12,\frac32]$, by the estimate \eqref{normal}, we get
\begin{equation}\label{tata}\begin{aligned}&\left\|\frac{ G(p)-G(1)}{p-1}+(k+1)! \p_\nu v_{0} + k! \p_\nu v_1\right\|_{H^{-\frac 1 2}(\p\Om)}=|p-1|^{-1}\|\partial_\nu V_p\|_{H^{-\frac 1 2}(\p\Om)}\\
\leq & C|p-1|^{-1}\|V_p\|_{H_\Delta(\Om)}\leq C|p-1|^{-1}(\|f_p\|_{L^2(\Om)}+\|g_p\|_{H^{\frac 1 2}(\p\Om)}),\end{aligned}\end{equation}
where $C>0$ depends only on $\Om$ that may change from line to line.
Thus the proof will be completed if
the following two relations hold:
\begin{align}
    &\lim_{p\to1}|p-1|^{-1}\|f_p\|_{L^2(\Om)}=\lim_{p\to1}|p-1|^{-1}\left\|(1-p^{\alpha})\widehat{U}(p,\cdot) + (p-1)k!\alpha v_0\right\|_{L^2(\Om)}=0,\label{eq:V:inside}\\
    &\lim_{p\to1} |p-1|^{-1}\|g_p\|_{H^{\frac 1 2}(\p\Om)}= \lim_{p\to1}\left|k!\left(\frac{p^{-k-1}-1}{p-1} + (k+1)\right)\right|\left\| e^{x\cdot\omega_0}\right\|_{H^{\frac 1 2}(\p\Om)}=0.\label{eq:V:boundary}
\end{align}
The relation \eqref{eq:V:boundary} follows directly and we only need to prove \eqref{eq:V:inside}. To this end, fix $p>0$, and let $W(p)=\widehat{U}(p,\cdot)-\widehat{U}(1,\cdot)\in H_\Delta(\Om)$. Note that $W(p)$ solves
$$
\left\{
\begin{aligned}
-\Delta W(p)+ W(p) &= (1-p^{\alpha(x)})\widehat{U}(p,\cdot) ,\quad \mbox{in }\Om,\\
W(p) &=k!e^{x\cdot\omega_0}(p^{-k-1}-1) ,\quad \mbox{on }\partial\Omega.
\end{aligned}
\right.
$$
Then, there exists $C>0$ depending only on $\Omega$ such that
$$\|W(p)\|_{H_\Delta(\Om)}\leq C(\|(1-p^{\alpha(x)})\widehat{U}(p,\cdot)\|_{L^2(\Om)}+|p^{-k-1}-1|\|e^{x\cdot\omega_0}\|_{H^{\frac 1 2}(\p\Om)}).$$
By combining this estimate with \eqref{l1a} and \eqref{l2a}, we deduce
$$\|W(p)\|_{H_\Delta(\Om)}\leq C(|p-1|+|p^{-k-1}-1|\|e^{x\cdot\omega_0}\|_{H^{\frac 1 2}(\p\Om)}),\quad \forall p\in[\tfrac12,\tfrac32].$$
Thus,  we obtain
\begin{equation}\label{tr}\lim_{p\to1}\|\widehat{U}(p,\cdot)-\widehat{U}(1,\cdot)\|_{H_\Delta(\Omega)}=  \lim_{p\to1}\|W(p)\|_{H_\Delta(\Omega)}=0.\end{equation}
By noting the identity $\widehat{U}(1,\cdot)=k!v_0$ and \eqref{l1b}, we get for $p\in[\frac12,\frac32]$,
\begin{align*}
&|p-1|^{-1}\left\|(1-p^{\alpha})\widehat{U}(p,\cdot) + (p-1)k!\alpha v_0\right\|_{L^2(\Om)}\\
=&\left\|\frac{1-p^{\alpha}}{1-p}(\widehat{U}(p,\cdot)-\widehat{U}(1,\cdot)) + \left(-\alpha+\frac{1-p^{\alpha}}{1-p}\right) \widehat{U}(1,\cdot)\right\|_{L^2(\Om)}\\
\leq& \left\|\frac{1-p^{\alpha}}{1-p}(\widehat{U}(p,\cdot)-\widehat{U}(1,\cdot))\right\|_{L^2(\Om)}+\left\|\left(-\alpha+\frac{1-p^{\alpha}}{1-p}\right) \widehat{U}(1,\cdot)\right\|_{L^2(\Om)}\\
\leq &(\|\alpha\|_{L^\infty(\Om)}+C|p-1|)\left\|\widehat{U}(p,\cdot)-\widehat{U}(1,\cdot)\right\|_{L^2(\Om)}+C|p-1|\|\widehat{U}(1,\cdot)\|_{L^2(\Om)}.
\end{align*}
By combining this estimate with \eqref{tr},  we obtain \eqref{eq:V:inside}, which directly implies \eqref{tata}. Thus we have $G'(1)=-(k+1)! \p_\nu v_{0} - k! \p_\nu v_1$ in the sense of Definition \ref{def:gen:derivative} with $X=H^{-\frac{1}{2}}(\partial\Omega)$.
\end{proof}

\begin{remark}\label{remark:v0:explicit}
Note that $v_0(x)=e^{x\cdot\omega_0}$ is independent of $p$ and has an explicit form satisfying \eqref{eq:bvp:v0}.
\end{remark}

The next lemma gives an important orthogonality identity.
\begin{lemma}\label{lemma:uniqueness:main:identity}
Let $\alpha^i$ with $i\in\{1,2\}$, and $g$ and $I$ satisfy Assumptions \ref{ass:alpha:general} and \ref{ass:g:exptype}, respectively.
Let $U=U^i$ with $i\in\{1,2\}$ be the solutions to problem \eqref{eq:2D3D:ibvp} with $\alpha=\alpha^i$ satisfying
\begin{equation}\label{l3a}
\p_\nu U^1(t,x)=\p_\nu U^2(t,x),\quad (t,x)\in I\times\p\Om.
\end{equation}
Then we have
\begin{equation}\label{eqn:orthogonal}
\int_\Om (\alpha^1(x)-\alpha^2(x)) e^{x\cdot(\omega+\omega_0)}\,{\rm d} x = 0,\quad\forall \omega\in\mathbb{S}^{d-1}.
\end{equation}
\end{lemma}
\begin{proof}
Theorem \ref{theorem:analyticity} and the estimate \eqref{normal} imply that, for $j=1,2$, the map $t\mapsto \p_\nu U^j(t,\cdot)|_{\p\Om}$ is analytic from $\mathbb{R}_+$ to $H^{-\frac 1 2}(\p\Om)$. Then, condition \eqref{l3a} and the isolated zero principle for analytic functions imply
\begin{equation}\label{l3b}
\p_\nu U^1(t,x)=\p_\nu U^2(t,x),\quad (t,x)\in \mathbb R_+\times\p\Om.
\end{equation}
By Theorem \ref{theorem:analyticity} and \eqref{normal} again, we deduce that, for all $p>0$, $t\mapsto e^{-pt}\p_\nu U^j(t,\cdot)|_{\p\Om}\in L^1(\mathbb R_+;H^{-\frac 1 2}(\p\Om))$. Therefore, the Laplace transform with respect to $t\in\mathbb R_+$ on \eqref{l3b} gives
\begin{equation}\label{l3c}G^1(p)=\int_0^{\infty}e^{-pt}\p_\nu U^1(t,\cdot)|_{\p\Om}\d t=\int_0^{\infty}e^{-pt}\p_\nu U^2(t,\cdot)|_{\p\Om}\d t=G^2(p)
,\quad p>0,\end{equation}
with  $G^j(p):=\p_\nu\widehat{U}^j(p,\cdot)|_{\p\Om}$ with $j=1,2$ and $p>0$.
For $j=1,2$, let $v_0$ and $v_1^j$ be defined by problem \eqref{eq:bvp:v0} with $\alpha=\alpha^j$.
By Theorem \ref{theorem:flux:freqisone:asympt:general}, $G^j(p)$ is differentiable at $p=1$ with
$$(G^j)'(1)=-(k+1)! \p_\nu v_{0} - k! \p_\nu v_1^j,\quad j=1,2.$$
By combining this with \eqref{l3c}, we obtain $(G^1)'(1)=(G^2)'(1)$, or equivalently,
\begin{equation}\label{eq:fullmeas:identity}
\p_\nu v_1^1(x)=\p_\nu v_1^2(x),\quad \forall x\in\p\Om.
\end{equation}
Fix $\omega\in\mathbb{S}^{d-1}$ and let the auxiliary function $z_\omega\in H_\Delta(\Om)$ be the weak solution to
\begin{equation}\label{eq:bvp:w:aux}
\left\{
\begin{aligned}
-\Delta z_\omega + z_\omega&=0,\quad\mbox{in }\Om,\\
z_\omega&=e^{x\cdot\omega},\quad\mbox{on }\partial\Omega.
\end{aligned}
\right.
\end{equation}
By the uniqueness of solutions of problem \eqref{eq:bvp:w:aux}, we have  $z_\omega(x)=e^{x\cdot\omega}$ for all $x\in\Om$.
Meanwhile, let $v=v_1^1-v_1^2$. Then it satisfies
$$
\left\{\begin{aligned}
-\Delta v + v &= (\alpha^1-\alpha^2) v_0,\quad \mbox{in }\Om,\\
v&= 0,\quad \mbox{on } \partial\Omega.
\end{aligned}\right.
$$
By multiplying $z_\omega$ on both sides of the equation,
integrating over the domain $\Omega$ and integrating by parts, we obtain the identity
$$\int_\Om \nabla v\cdot\nabla z_\omega \,{\rm d}x + \int_\Om vz_\omega \,{\rm d}x- \left\langle\p_\nu v, z_\omega\right\rangle_{H^{-\frac 1 2}(\p\Om),H^{\frac 1 2}(\p\Om)} = \int_\Om (\alpha^1-\alpha^2)v_0 z_\omega \,{\rm d}x.$$
By the weak formulation of problem \eqref{eq:bvp:w:aux} and  the identity \eqref{eq:fullmeas:identity}, we arrive at
$$\int_\Om (\alpha^1-\alpha^2) v_0 z_\omega\,{\rm d} x = - \left\langle\p_\nu v,z_\omega\right\rangle_{H^{-\frac 1 2}(\p\Om),H^{\frac 1 2}(\p\Om)} =0.$$
Using the explicit relations $v_0=e^{x\cdot\omega_0}$ and $z_\omega=e^{x\cdot\omega}$, we obtain the desired assertion \eqref{eqn:orthogonal}.
\end{proof}

The next lemma gives a holomorphic extension of the identity \eqref{eqn:orthogonal}. Let $\bm\varphi_{d-2}=(\varphi_1,\dots,\varphi_{d-2})$.
\begin{lemma}\label{lemma:exponential:in:integralkernel}
    Let $\{\mathrm{e}_j\}_{j=1}^d$ be an orthonormal basis of $\mathbb{R}^d$, $\beta\in L^\infty(\Om;\mathbb{C})$,
    and for all $(\theta,\bm{\varphi}_{d-2})\in\mathbb{C}^{d-1}$, let
\begin{equation}\label{eqn:omega} \omega(\theta,\bm{\varphi}_{d-2}):=(\cos\theta)\mathrm{e}_1+\sum_{i=2}^{d-1}\left(\sin\theta\left(\prod_{j=1}^{i-2}\sin\varphi_j\right)\cos\varphi_{i-1}\right)\mathrm{e}_{i}+\left(\sin\theta\left(\prod_{j=1}^{d-2}\sin\varphi_j\right)\right)\mathrm{e}_{d}.
\end{equation}
When $d=2$, let $\omega(\theta)=(\cos\theta) {\mathrm{e}}_1 + (\sin\theta){\mathrm{e}}_2$.  For any $\bm\varphi_{d-2}\in\mathbb{R}^{d-2}$, the map $W:\mathbb{C}^{d-1}\to\mathbb{C}$ defined by
\begin{equation}\nonumber
    W(\theta,\bm\varphi_{d-2}):= \int_{\Om}\beta(x)e^{ x\cdot\omega(\theta,\bm\varphi_{d-2})}\,{\rm d}x,\quad\forall(\theta,\bm\varphi_{d-2})\in\mathbb{C}^{d-1},
    \end{equation}
    is holomorphic in each entry of $(\theta,\bm\varphi_{d-2})$ in $\mathbb{C}$ with the other entries fixed.
\end{lemma}
\begin{proof}
Fix any $\bm\varphi_{d-2}\in\mathbb{C}^{d-2}$, and let  $f(\theta)=\omega(\theta,\bm\varphi_{d-2})$. Then its derivative $f'$ is given by
$$f'(\theta)=(-\sin\theta)\mathrm{e}_1+\sum_{i=2}^{d-1}\left(\cos\theta\left(\prod_{j=1}^{i-2}\sin\varphi_j\right)\cos\varphi_{i-1}\right)\mathrm{e}_{i}+\left(\cos\theta\left(\prod_{j=1}^{d-2}\sin\varphi_j\right)\right)\mathrm{e}_{d}.$$
We prove that the map $\theta\mapsto W(\theta,\bm\varphi_{d-2})$ is holomorphic in $\mathbb{C}$. Let
$\widetilde{W}(\theta):= \int_{\Om}\beta(x)(x\cdot  f'(\theta))e^{x\cdot f(\theta)}\,{\rm d}x.$
By H\"{o}lder's inequality, for all $\theta,\varepsilon\in\mathbb{C}$, we have
\begin{align}
    \nonumber&\hspace{4mm}
        \left|W(\theta+\varepsilon,\bm\varphi_{d-2})-W(\theta,\bm\varphi_{d-2})-\varepsilon \widetilde {W}(\theta)\right|\\
        \nonumber&=\left|\int_\Om \beta(x)\left[e^{x\cdot f(\theta+\varepsilon)}-e^{x\cdot f(\theta)}-\varepsilon(x\cdot f'(\theta))e^{x\cdot f(\theta)}\right]\,{\rm d}x\right|\\
        &\le \|\beta(x)\|_{L^\infty(\Om)} \left\|e^{x\cdot f(\theta+\varepsilon)}-e^{x\cdot f(\theta)}-\varepsilon(x\cdot f'(\theta))e^{x\cdot f(\theta)}\right\|_{L^1(\Om)}.\label{eq:Holder:pf:holomorphic}
    \end{align}
    Fix any $\theta_0\in\mathbb{C}$. We claim that $L^1(\Omega)$ term with $\theta=\theta_0$ in \eqref{eq:Holder:pf:holomorphic} is of class $o(|\varepsilon|)$  as $\varepsilon\to 0$. Since each component of $f$ is holomorphic in $\mathbb{C}$, there exists $\widetilde{f}:\mathbb{C}\to\mathbb{C}^{d}$, whose components are holomorphic in $\mathbb{C}$, such that
$$f(\theta_0+\varepsilon)-f(\theta_0)-\varepsilon f'(\theta_0)=\varepsilon^2 \widetilde{f}(\varepsilon),\quad\forall \varepsilon\in\mathbb{C}.$$
    Similarly, since $e^z$ is holomorphic in $\mathbb{C}$, there exists  $f_e:\mathbb{C}\to\mathbb{C}$  holomorphic in $\mathbb{C}$ satisfying $e^z-1-z=z^2f_e(z)$ for all $z\in\mathbb{C}$.
    Hence, for each $x\in\Om$, with $z\equiv x\cdot(\varepsilon f'(\theta_0)+\varepsilon^2\widetilde{f}(\varepsilon))$, we have
    \begin{align}
        \nonumber e^{x\cdot f(\theta_0+\varepsilon)}-e^{x\cdot f(\theta_0)}&=\left[e^{x\cdot( f(\theta_0+\varepsilon)- f(\theta_0))}-1\right]e^{x\cdot f(\theta_0)} =\left[e^{x\cdot(\varepsilon f'(\theta_0)+\varepsilon^2 \widetilde{f}(\varepsilon))}-1\right]e^{x\cdot f(\theta_0)}\\
        &=\left[z + z^2 f_e(z) \right]e^{x\cdot f(\theta_0)}
        =\varepsilon(x\cdot f'(\theta_0))e^{x\cdot f(\theta_0)} + \varepsilon^2 R(x,\varepsilon),\label{eq:remainder:pf:holomorphic}
    \end{align}
with $R(x,\varepsilon)=[x\cdot \widetilde{f}(\varepsilon)+z^2  f_e(z)]e^{x\cdot f(\theta_0)}$ for all $(x,\varepsilon)\in\overline{\Om}\times\mathbb{C}$.
Since $\widetilde f$ and $f_e$ are entire functions on $\mathbb{C}$, $R$ is continuous on the compact set $\overline{\Om}\times \{z\in\mathbb{C}\,:\,|z|\le1\}$, and we have $\sup\{|R(x,\varepsilon)|\,:\,x\in\overline{\Om},\,\varepsilon\in\mathbb{C},\,|\varepsilon|\le1\}<\infty$. This directly implies
\begin{equation}\label{eq:remainder:pf:holomorphic:2}
        \|\varepsilon^2 R(x,\varepsilon)\|_{L^1(\Om)}\le |\varepsilon|^2 |\Om| \sup\{|R(x,\varepsilon)|\,:\,x\in\overline{\Om},\,\varepsilon\in\mathbb{C},\,|\varepsilon|\le1\} = o(|\varepsilon|)\quad\mbox{as }\varepsilon\to0.
    \end{equation}
Combining \eqref{eq:Holder:pf:holomorphic}, \eqref{eq:remainder:pf:holomorphic} and \eqref{eq:remainder:pf:holomorphic:2} gives that $\widetilde {W}(\theta_0)$ is the complex derivative of $W(\theta,\bm\varphi_{d-2})$ at $\theta=\theta_0$.
Since the  argument holds for any $\theta_0\in\mathbb{C}$, the map $\theta\mapsto W(\theta,\bm\varphi_{d-2})$ is holomorphic in $\mathbb{C}$.
The proof of complex differentiability with respect to the components of $(\bm\varphi_{d-2})$ is similar and thus omitted.
\end{proof}
\begin{remark}\label{remark:extended:identity}
Under the assumptions of Lemma \ref{lemma:uniqueness:main:identity}, by Lemma \ref{lemma:exponential:in:integralkernel} with $\beta(x) = (\alpha^1(x)-\alpha^2(x))e^{x\cdot\omega_0}$, for any fixed $\bm\varphi_{d-2}\in\mathbb{C}^{d-2}$, the map $W:\mathbb{C}\to\mathbb{C}$ defined by
    $$W(\theta):= \int_{\Om}(\alpha^1(x)-\alpha^2(x))e^{x\cdot(\omega(\theta,\bm\varphi_{d-2})+\omega_0)}\,{\rm d}x,\quad\forall\theta\in\mathbb{C}$$
    is holomorphic.
    Note that $\omega(\theta,\bm \varphi_{d-2})\in\mathbb{S}^{d-1}$ for all $(\theta,\bm\varphi_{d-2})\in\mathbb{R}^{d-1}$.
    Hence, for any  $\bm \varphi_{d-2}\in\mathbb{R}^{d-2}$, Lemma \ref{lemma:uniqueness:main:identity} implies that the holomorphic function $W$ satisfies $W(\theta)=0$ all $\theta\in\mathbb{R}$ so that $W(\theta)=0$ for all $\theta\in\mathbb{C}$ by the unique continuation property. That is, we have
\begin{equation}\label{eq:extended:identity}
        \int_{\Om}(\alpha^1(x)-\alpha^2(x))e^{x\cdot(\omega(\theta,\bm \varphi_{d-2})+\omega_0)}\,{\rm d}x=0,\quad\forall\theta\in\mathbb{C}.
    \end{equation}
\end{remark}

In Sections \ref{sec:inverse} and \ref{section:polygons},
we carefully choose an orthonormal basis $\{\mathrm{e}_j\}_{j=1}^d$ of $\mathbb{R}^d$ that determines $\omega(\theta,\bm \varphi_{d-2})$ in \eqref{eqn:omega}  and $\widetilde{\theta}\in\mathbb{R}$ so that the identity \eqref{eq:extended:identity} for $\theta$ on the half-line $\{\widetilde{\theta}-{\rm i}R\,:\,R>0\}$ implies $\alpha^1=\alpha^2$.
This analysis depends on the type of the unknown perturbation in $\alpha$.

We also use the following technical lemma.
\begin{lemma}\label{lemma:sphere:and:hyperplane}
    Let $d\ge2$, $n\in\mathbb{N}$, $a_i\in\mathbb{R}^d\backslash\{0\}$ and $b_i\in\mathbb{R}^d$ for each $i\in [n]$. Then the set
    $$\{\omega\in\mathbb{S}^{d-1}\,:\,\omega\cdot a_i \neq b_i,\ \forall i\in[n]\}$$
    is open and dense with respect to the relative topology on $\mathbb{S}^{d-1}\subset\mathbb{R}^d$.
\end{lemma}
\begin{proof}
For each $i\in[n]$, let $H_i$ be the hyperplane defined by
$H_i:=\{\omega\in\mathbb{R}^{d}\,:\,\omega\cdot a_i = b_i\}.$ Each $H_i$ is closed in $\mathbb{R}^d$, so its complement is open, and thus, $\mathbb{S}^{d-1}\backslash \bigcup_{i=1}^n H_i$ is open with respect to the relative topology on $\mathbb{S}^{d-1}\subset\mathbb{R}^d$.
    For each $i\in[n]$, by the hypothesis $a_i\ne0$, the set $\mathbb{S}^{d-1}\cap H_i$ is of one of the three types: empty, singleton or a sphere of Hausdorff dimension $d-2$, and thus has surface measure zero on $\mathbb{S}^{d-1}$.
    Thus, the set $\mathbb{S}^{d-1}\backslash \bigcup_{i=1}^n H_i$ is dense in $\mathbb{S}^{d-1}$.
\end{proof}

\section{Spherical inclusions}\label{sec:inverse}

In this section we prove the uniqueness for spherical inclusions in Theorem \ref{theorem:balls}.

\subsection{Proof of Theorem \ref{theorem:balls}}

The proof of Theorem \ref{theorem:balls} relies on two technical lemmas, whose proofs are given in Section \ref{lemmas:balls}. The notation $J_\nu$ denotes the Bessel function of order $\nu$; see Definition \ref{definition:Bessel} below for details.
\begin{lemma}\label{lemma:Bessel}
    For all $d\ge 2$ and $y\in\mathbb{R}^{d}\backslash\{0\}$, we have
    $$\int_{B_{r}(0)} e^{ x\cdot y}\,{\rm d}x =
        (2\pi r|y|^{-1})^{\frac{d}{2}} e^{-\frac{d}{4}\pi{\rm i}} J_{\frac{d}{2}}({\rm i}r|y|).$$
\end{lemma}
\begin{lemma}
\label{lemma:main:linindep:general}
Fix any $d\ge2$ and $\omega_0\in\mathbb{S}^{d-1}$.
Let $\{(r_i,x_i)\}_{i=1}^N\subset\mathbb{R}_{+}\times\mathbb{R}^d$ be a finite set of distinct pairs.
If the constants $\{C_i\}_{i=1}^N\subset \mathbb{R}$ satisfy
$$\sum_{i=1}^N C_i e^{x_i\cdot\omega}\int_{B_{r_i}(0)} e^{ x\cdot (\omega+\omega_0)}\,{\rm d}x=0,\quad\forall \omega\in\mathbb{S}^{d-1},$$
then $C_i=0$ for all $i\in [N]$.
\end{lemma}

\begin{proof}[Proof of Theorem \ref{theorem:balls}]
Fix any $\alpha_{\rm in}$ satisfying Assumption \ref{ass:alpha:general} and suppose that $\alpha=\alpha^i$ satisfies Assumption \ref{ass:alpha} for $i\in\{1,2\}$ with the parameters $N=N^i$, $r_j=r_j^i$, $x_j=x_j^i$ and $\alpha_j=\alpha_j^i$:
\begin{equation}\label{eq:alpha:spheres}
    \alpha^i(x) = \alpha_{\rm in}(x)+\sum_{j=1}^{N^i} \alpha_j^i \mathds{1}_{B_{r_j^i}(x_j^i)}(x),\quad\forall x\in\Om,\,i\in\{1,2\}.
\end{equation}
By combining \eqref{eq:alpha:spheres} with Lemma \ref{lemma:uniqueness:main:identity}, we obtain
$$\sum_{j=1}^{N^1} \alpha_j^1 \int_{B_{r_j^1}(x_j^1)} e^{x\cdot(\omega+\omega_0)}\,{\rm d}x=\sum_{j=1}^{N^2} \alpha_j^2 \int_{B_{r_j^2}(x_j^2)}e^{x\cdot(\omega+\omega_0)}\,{\rm d}x.$$
Using the change of variables $y={x-x_j^i}$ for each $i\in\{1,2\}$ and $j\in[ N^i]$, we arrive at
\begin{equation}\label{eq:main:identity:balls:general}
    \sum_{j=1}^{N^1} \alpha_j^1 e^{x_j^1\cdot\omega_0}e^{x_j^1\cdot\omega}\int_{B_{r_j^1}(0)} e^{y\cdot(\omega+\omega_0)}\,{\rm d}y=\sum_{j=1}^{N^2} \alpha_j^2 e^{x_j^2\cdot\omega_0}e^{x_j^2\cdot\omega}\int_{B_{r_j^2}(0)}e^{y\cdot(\omega+\omega_0)}\,{\rm d}y.
\end{equation}
Now for each $i\in\{1,2\}$, we can reduce the collection of tuples $\{(\alpha_j^i,r_j^i,x_j^i)\}_{j=1}^{N^i}$ satisfying \eqref{eq:alpha:spheres} by
\begin{itemize}
    \item If $\alpha_j^i=0$ for some $j\in[N^i]$, then we remove the tuple $(\alpha_j^i,r_j^i,x_j^i)$ from the collection;
    \item If $(r_j^i,x_j^i)=(r_k^i,x_k^i)$ for some $j\ne k\in[N^i]$, then for such $j$ and $k$, we merge the two tuples $(\alpha_j^i,r_j^i,x_j^i)$ and $(\alpha_k^i,r_k^i,x_k^i)$ into $(\alpha_j^i+\alpha_k^i,r_j^i,x_j^i)$. 
\end{itemize}
Thus we may assume in \eqref{eq:alpha:spheres} the additional properties $\alpha_j^i\ne0$ and $(r_j^i,x_j^i)\ne(r_k^i,x_k^i)$ for all $j\ne k\in[N^i]$.
Using the identity \eqref{eq:main:identity:balls:general}, we shall prove
\begin{equation}\label{eq:identify:tuples:balls}
    \{(\alpha_j^1,r_j^1,x_j^1)\}_{j=1}^{N^1}=\{(\alpha_j^2,r_j^2,x_j^2)\}_{j=1}^{N^2}
\end{equation}
First, we prove  $\{(r_j^1,x_j^1)\}_{j=1}^{N^1}=\{(r_j^2,x_j^2)\}_{j=1}^{N^2}$. Fix any $k\in[N^1]$ and suppose $(r_k^1,x_k^1)\not\in\{(r_j^2,x_j^2)\}_{j=1}^{N^2}$. Then, by Lemma \ref{lemma:main:linindep:general} and the identity \eqref{eq:main:identity:balls:general}, we must have $\alpha_k^1e^{x_k^1\cdot\omega_0}=0$, which contradicts the hypothesis.
This proves $(r_k^1,x_k^1)\in\{(r_i^2,x_i^2)\}_{i=1}^{N^2}$ for all $k\in[N^1]$.
Similarly, we also have $(r_k^2,x_k^2)\in\{(r_j^1,x_j^1)\}_{j=1}^{N^1}$ for all $k\in[N^2]$.
Thus, we arrive at $\{(r_j^1,x_j^1)\}_{j=1}^{N^1}=\{(r_j^2,x_j^2)\}_{j=1}^{N^2}$.
Since there holds $(r_j^i,x_j^i)\ne(r_k^i,x_k^i)$ for all $j\ne k\in[N^i]$, we deduce $N^1=N^2=:N$.
Now we rearrange the tuples so that $(r_j^1,x_j^1)=(r_j^2,x_j^2)=:(r_j,x_j)$ for all $j\in[N]$.
By Lemma \ref{lemma:main:linindep:general} and the relation \eqref{eq:main:identity:balls:general} again, we obtain
$$\alpha_j^1 e^{x_j\cdot\omega_0}=\alpha_j^2 e^{x_j\cdot\omega_0},\quad\forall j\in[N].$$
Thus, we have $\alpha_j^1=\alpha_j^2$ for all $j$, which proves \eqref{eq:identify:tuples:balls}.
This implies $\alpha^1(x)=\alpha^2(x)$ for all $x\in\Om$.
\end{proof}

\subsection{Proofs of Lemmas \ref{lemma:Bessel} and \ref{lemma:main:linindep:general}}\label{lemmas:balls}

We first recall preliminary properties of  the Gamma and Bessel functions.

\begin{lemma}[{\cite[(1.7.3)]{Szego:1975:OP}}]
\label{lemma:Gamma:properties}
The Gamma function $\Gamma(z)$ satisfies Legendre duplication formula $$\Gamma(z)\Gamma(z+\tfrac12)=2^{-2z+1}\sqrt{\pi}\Gamma(2z),\quad \forall z>0.$$
\end{lemma}

\begin{lemma}
\label{lemma:Beta:Gamma:relation}
For all $a,b>0$, there holds
$\int_{0}^{\frac{\pi}{2}} \sin^{2a-1}\theta\cos^{2b-1}\theta \,{\rm d}\theta={\frac{1}{2}}\frac{\Gamma(a)\Gamma(b)}{\Gamma(a+b)}.$
\end{lemma}
\begin{proof}
From {\cite[(2.13)]{Artin:1964:Gamma}}, there holds $\int_{0}^{1} t^{a-1}(1-t)^{b-1} \,{\rm d}t=\frac{\Gamma(a)\Gamma(b)}{\Gamma(a+b)}.$ Then with the substitution $t=\sin^2\theta$ (i.e., ${\rm d}t=2\sin \theta \cos\theta {\rm d}\theta$), we obtain the desired identity.
\end{proof}

\begin{definition}
\label{definition:Bessel}
The Bessel function $J_\nu(z)$ of the first kind of order $\nu\in\mathbb{R}$ with $\nu\ge 0$ is defined by
\begin{align*}
J_\nu(z)=\sum_{n=0}^\infty \frac{(-1)^n}{n!\Gamma(\nu+n+1)}\left(\frac{z}{2}\right)^{\nu+2n},\quad\forall z\in\mathbb{R}.
\end{align*}
If $\nu$ is an integer, $J_\nu(z)$  extends to $z\in\mathbb{C}$ holomorphically.
If $\nu+\frac12$ is an integer, $J_\nu(z)$ extends to $z\in\mathbb{C}\backslash\{z\in\mathbb{R}\,:\,z\le0\}$ holomorphically via $z^{\frac{1}{2}}=e^{\frac{1}{2}\log z}$ $($see, e.g., \cite[Theorem 6.1]{Stein:2003:ComplexAnalysis} for the existence of the holomorphic branch of logarithm$)$.
\end{definition}

\begin{lemma}[{\cite[(1.71.9)]{Szego:1975:OP}}]
\label{lemma:Bessel:asymptotic}
Fix $\nu\ge0$ and $\delta>0$. The Bessel function $J_\nu (z)$ satisfies the following asymptotic relation for $\arg (z)\in(-\frac{\pi}{2}+\delta,\frac{3\pi}{2}-\delta)$ as $|z|\to\infty$:
$$e^{\frac{\nu}{2}\pi{\rm i}}J_\nu(e^{-\frac{\pi}{2}{\rm i}}z)=(2\pi z)^{-\frac12}e^z(1+O(|z|^{-1})) + (2\pi z)^{-\frac12}e^{-z+(\nu+\frac12)\pi{\rm i}}(1+O(|z|^{-1})).$$
\end{lemma}

Now we prove Lemmas \ref{lemma:Bessel} and \ref{lemma:main:linindep:general}, which are the main tools in the proof of Theorem \ref{theorem:balls}.
\begin{proof}[Proof of Lemma \ref{lemma:Bessel}]
We prove the cases $d=2$ and $d\geq 3$ separately.\\
Case (i) $d=2$. Let $\omega_1=\frac{y}{|y|}$ and $\omega_2\in\mathbb{S}^1$ such that $\{\omega_1,\omega_2\}$ is an orthogonal basis of $\mathbb{R}^2$.
Changing variables to the polar coordinates $(\tau,\theta)$ defined by $x=(\tau\cos\theta)\omega_1+(\tau\sin\theta)\omega_2$ and the definition $\cosh s = \frac{1}{2}(e^s+e^{-s})$ for all $s\in\mathbb{R}$ give
    \begin{align*}
         \int_{B_{r}(0)} e^{ x\cdot y}\,{\rm d}x&= \int_{-\pi}^{\pi} \int_0^r e^{|y|\tau\cos\theta}\tau \,{\rm d}\tau\,{\rm d}\theta= 4\int_{0}^{\frac\pi2} \int_0^r \cosh(|y|\tau\cos\theta)\tau \,{\rm d}\tau\,{\rm d}\theta.
    \end{align*}
The Taylor expansion  $\cosh s = \sum_{m=0}^\infty \frac{s^{2m}}{\Gamma(2m+1)}$ (for $s\in\mathbb{R}$),
and term-by-term integration
and Lemma \ref{lemma:Beta:Gamma:relation} (with $a=\frac{1}{2}$ and $b=m+\frac{1}{2}$) give
\begin{align*}
\int_{B_{r}(0)} e^{ x\cdot y}\,{\rm d}x&=4\sum_{m=0}^\infty  \frac{|y|^{2m}}{\Gamma(2m+1)} \int_{0}^{\frac{\pi}{2}} \cos^{2m}\theta\,{\rm d}\theta\int_0^r \tau^{2m+1}\,{\rm d}\tau\\ &=4\sum_{m=0}^\infty  \frac{|y|^{2m}}{\Gamma(2m+1)} \frac{\Gamma(\frac12)\Gamma(m+\frac12)}{2\Gamma(m+1)}\frac{r^{2m+2}}{2m+2}.
\end{align*}
Now by the Legendre duplication formula in Lemma \ref{lemma:Gamma:properties} (with $z=m+\frac12$),
\begin{equation}\label{eqn:id-Gamma} \Gamma(m+\tfrac12)=2^{-2m}\sqrt{\pi}\frac{\Gamma(2m+1)}{\Gamma(m+1)}.
\end{equation}
From the preceding two identities, $\Gamma(\frac12)=\sqrt{\pi}$ and the definition of $J_\nu(z)$, we obtain
    $$\int_{B_{r}(0)} e^{ x\cdot y}\,{\rm d}x = {\pi}r^2 \sum_{m=0}^\infty   \frac{2^{-2m}|y|^{2m}r^{2m}}{\Gamma(m+1)\Gamma(m+2)}=-2{\rm i}{\pi}r |y|^{-1} J_1({\rm i}r|y|).$$

\noindent
Case (ii) $d\ge3$. Let $\omega_1=\frac{y}{|y|}$ and fix any $\{\omega_i\}_{i=2}^d\subset\mathbb{S}^{d-1}$ such that $\{\omega_i\}_{i=1}^d$ forms an orthonormal basis of $\mathbb{R}^d$. We change variables to the hyperspherical coordinates $(\tau,\theta,\bm{\varphi}_{d-2})$ for $\mathbb{R}^d$ defined by
    $$x=(\tau\cos\theta)\omega_1+\sum_{i=2}^{d-1}\left(\tau\sin\theta\left(\prod_{j=1}^{i-2}\sin\varphi_j\right)\cos\varphi_{i-1}\right)\omega_{i}+\left(\tau\sin\theta\left(\prod_{j=1}^{d-2}\sin\varphi_j\right)\right)\omega_{d}.$$
The volume element ${\rm d}x$ is given by ${\rm d}x=(\tau^{d-1}\,{\rm d}\tau )(\sin^{d-2}\theta\,{\rm d}\theta)\prod_{i=1}^{d-2}(\sin^{d-i-2}\varphi_i\,{\rm d}\varphi_i).$
Thus, with $|\mathbb{S}^{d-2}|$ being the surface area of the unit sphere in $\mathbb{R}^{d-1}$, there holds
    \begin{align}
         \int_{B_{r}(0)} e^{ x\cdot y}\,{\rm d}x&= \int_0^{\pi} \int_0^r e^{|y|\tau\cos\theta}\tau^{d-1} \sin^{d-2}\theta\,{\rm d}\tau\,{\rm d}\theta \int_{-\pi}^{\pi}\,{\rm d}\varphi_{d-2}\prod_{i=1}^{d-3}\int_{0}^{\pi}\sin^{d-i-2}\varphi_i\,{\rm d}\varphi_i\nonumber\\
         &=|\mathbb{S}^{d-2}|\int_0^{\pi} \int_0^r e^{|y|\tau\cos\theta}\tau^{d-1} \sin^{d-2}\theta\,{\rm d}\tau\,{\rm d}\theta.\label{eq:3D:integral}
    \end{align}
The Taylor expansion of $e^z$, term-by-term integration
and Lemma \ref{lemma:Beta:Gamma:relation} (with $a=\frac{d-1}{2}$ and $b=m+\frac{1}{2}$) give
    \begin{align*}
        &\int_0^{\pi} \int_0^r e^{|y|\tau\cos\theta}\tau^{d-1} \sin^{d-2}\theta\,{\rm d}\tau\,{\rm d}\theta\\
        =&\sum_{m=0}^\infty\frac{|y|^{2m}}{\Gamma(2m+1)}\int_0^r \tau^{2m+d-1}\,{\rm d}\tau \int_0^{\pi} \cos^{2m}\theta\sin^{d-2}\theta\,{\rm d}\theta\\
        =&\sum_{m=0}^\infty\frac{|y|^{2m}}{\Gamma(2m+1)}\frac{r^{2m+d}}{2m+d}\frac{\Gamma(\frac{d-1}{2})\Gamma(m+\frac{1}{2})}{\Gamma(m+\frac{d}{2})}.
    \end{align*}
    By Legendre duplication formula (Lemma \ref{lemma:Gamma:properties} with $z=m+\frac12$) (or the identity \eqref{eqn:id-Gamma}), we have
    \begin{align}
       & \int_0^{\pi} \int_0^r e^{|y|\tau\cos\theta}\tau^{d-1} \sin^{d-2}\theta\,{\rm d}\tau\,{\rm d}\theta\nonumber\\
       =&\frac{\sqrt{\pi}}{2}\Gamma(\tfrac{d-1}{2})\sum_{m=0}^\infty\frac{2^{-2m}|y|^{2m}r^{2m+d}}{\Gamma(m+1)\Gamma(m+1+\frac{d}{2})}\nonumber\\
        =&\frac{\sqrt{\pi}}{2}\Gamma(\tfrac{d-1}{2})(2r|y|^{-1})^{\frac{d}{2}}e^{-\frac{d}{4}\pi{\rm i}} J_{\frac{d}{2}}({\rm i} r |y|).\label{eq:Taylor:Bessel}
    \end{align}
By combining \eqref{eq:3D:integral} and \eqref{eq:Taylor:Bessel} with $|\mathbb{S}^{d-2}|=\frac{2\pi^{\frac{d-1}{2}}}{\Gamma(\frac{d-1}{2})}$, we obtain the desired assertion for $d\ge3$.
\end{proof}

The next lemma extends the formula in Lemma \ref{lemma:Bessel} to complex variables in $\widetilde{K}=\{\theta\in\mathbb{C}\,:\,|\Re(\theta)|<\frac{\pi}{2}\}$.
\begin{lemma}\label{lemma:Bessel:complexversion}
    Let $\{\mathrm{e}_j\}_{j=1}^d$ be any orthonormal basis of $\mathbb{R}^d$ and let $\omega:\mathbb{C}^{d-1}\to\mathbb{C}^d$ be defined in \eqref{eqn:omega}. Fix any $\bm \varphi_{d-2}\in\mathbb{R}^{d-2}$.
    The following properties hold:
    \begin{itemize}
        \item[\rm(a)]
        $|\omega(\theta,\bm \varphi_{d-2})+\mathrm{e}_1|=2|\cos\tfrac{\theta}{2}|$ for all $\theta\in\mathbb{R}$.
        \item[\rm(b)] Let $r>0$, and $f(\theta)=2\cos\tfrac{\theta}{2}$ for all $\theta\in\mathbb{C}$. The functions $\theta\mapsto f(\theta)^{-\frac{d}{2}}$ and $\theta\mapsto J_{\frac{d}{2}}({\rm i}r f(\theta))$ are well defined and holomorphic in $\theta\in \widetilde{K}$. Moreover, there holds
        \begin{equation}\label{eq:extended:explicit:balls}
            \int_{B_{r}(0)} e^{x\cdot (\omega(\theta,\bm \varphi_{d-2})+\mathrm{e}_1)}\,{\rm d}x =
        (2\pi r)^{\frac{d}{2}} f(\theta)^{-\frac{d}{2}} e^{-\frac{d}{4}\pi{\rm i}} J_{\frac{d}{2}}({\rm i}rf(\theta)),\quad\forall\theta\in \widetilde{K}.
        \end{equation}
    \end{itemize}
\end{lemma}
\begin{proof}
Part (a) follows directly as
\begin{align}
\nonumber
&\quad |\omega(\theta,\bm \varphi_{d-2})+{\mathrm{e}}_1|=\left(\sum_{i=1}^d |(\omega(\theta,\bm\varphi_{d-2})+{\mathrm{e}}_1)\cdot{\mathrm{e}}_i|^2\right)^{\frac12}\\
\nonumber&=\left((\cos\theta+1)^2+\sum_{i=2}^{d-1}\left(\sin\theta\left(\prod_{j=1}^{i-2}\sin\varphi_j\right)\cos\varphi_{i-1}\right)^2+\left(\sin\theta\left(\prod_{j=1}^{d-3}\sin\varphi_j\right)\right)^2\right)^{\frac12}\\
\nonumber&=\left((\cos\theta+1)^2+\sin^2\theta\right)^{\frac12}=\left(2+2\cos\theta\right)^{\frac12}=\left(4\cos^2\tfrac{\theta}{2}\right)^{\frac12}=2|\cos\tfrac{\theta}{2}|,\quad\forall\theta\in\mathbb{R}.
\end{align}
Next we prove part (b). Let $\log$ be the holomorphic branch of the logarithm function defined on $\mathbb{C}_+:=\{z\in\mathbb{C}\,:\,\Re(z)>0\}$  \cite[Theorem 6.1]{Stein:2003:ComplexAnalysis}.
Let $J_{\frac{d}{2}}$ be defined on $\mathbb{C}\backslash\{x\in\mathbb{R}\,:\,x\le 0\}$ as in Definition \ref{definition:Bessel}.
We claim that both $f(\theta)^{-\frac{d}{2}}$ and $J_{\frac{d}{2}}({\rm i}rf(\theta))$ for $\theta\in \widetilde K$ are well-defined as holomorphic functions via
\begin{align*}
    f(\theta)^{-\frac{d}{2}}&=e^{-\tfrac{d}{4}\log(2+2\cos\theta)}\quad \mbox{and}\quad
    J_{\frac{d}{2}}({\rm i}rf(\theta))=J_{\frac{d}{2}}\left({\rm i}re^{\frac{1}{2}\log(2+2\cos\theta)}\right),
\end{align*}
which are obviously compatible with the definition $f(\theta)=2|\cos\tfrac{\theta}{2}|$ for $\theta\in\mathbb{R}$.
Note that there holds
$$\Re\left(2+2\cos\theta\right)=\Re\left(2+e^{{\rm i}\theta}+e^{-{\rm i}\theta}\right)=2+(e^{-\Im(\theta)}+e^{\Im(\theta)})\cos(\Re(\theta)),\quad\forall \theta\in\mathbb{C}.$$
For all $\theta\in\widetilde K$, we have $\cos(\Re(\theta))>0$, and thus
there holds $2+2\cos\theta\in \mathbb{C}_+$.
Hence $f(\theta)^{-\frac{d}{2}}$ is well-defined and holomorphic in $\theta\in \widetilde{K }$.
Also, we have
$${\rm i}re^{\frac{1}{2}\log (2+2\cos \widetilde K)}\subset{\rm i}re^{\frac{1}{2}\log \mathbb{C}_+}\subset\{z\in\mathbb{C}\,:\,|\arg(z)-\tfrac{\pi}{2}|<\tfrac{\pi}{4}\}\subset \mathbb{C}\backslash\{z\in\mathbb{R}\,:\,z\le 0\},$$
and so $J_{\frac{d}{2}}({\rm i}rf(\theta))$ is well-defined and holomorphic in $\theta\in \widetilde K$.
Thus both $f(\theta)^{-\frac{d}{2}}$ and $J_{\frac{d}{2}}({\rm i}r_i f(\theta))$ are holomorphic in $\theta\in \widetilde{K}$, so $
(2\pi r)^{\frac{d}{2}} f(\theta)^{-\frac{d}{2}} e^{-\frac{d}{4}\pi{\rm i}} J_{\frac{d}{2}}({\rm i}rf(\theta))$ is holomorphic in $\theta\in \widetilde{K}$.
By Lemma \ref{lemma:exponential:in:integralkernel}, $\int_{B_{r}(0)} e^{x\cdot (\omega(\theta,\bm \varphi_{d-2})+\mathrm{e}_1)}\,{\rm d}x$ is also holomorphic in $\theta\in\mathbb{C}$.
Moreover, by part (a) and Lemma \ref{lemma:Bessel}, the identity \eqref{eq:extended:explicit:balls} holds for all $\theta\in\mathbb{R}$.
Thus, by the unique continuation property of holomorphic functions, the identity \eqref{eq:extended:explicit:balls} holds for all $\theta\in \widetilde{K}$.
\end{proof}

Now we can state the proof of Lemma \ref{lemma:main:linindep:general}.
\begin{proof}[Proof of Lemma \ref{lemma:main:linindep:general}]
The proof employs Lemmas \ref{lemma:Bessel:asymptotic} and \ref{lemma:Bessel:complexversion}.
Set $\hat{\mathrm{e}}_1:=\omega_0$ and choose any $\hat{\mathrm{e}}_2,\dots,\hat{\mathrm{e}}_{d}\in\mathbb{S}^{d-1}$ so that $\{\hat{\mathrm{e}}_i\}_{i=1}^{d}$ forms an orthonormal basis of $\mathbb{R}^d$.
We define $\omega:\mathbb{C}^{d-1}\to\mathbb{C}^d$ by \eqref{eqn:omega} with $\mathrm{e}_i=\hat{\mathrm{e}}_i$ for all $i\in[N]$.
Since $\omega$ maps $\mathbb{R}^{d-1}$ to $\mathbb{S}^{d-1}$, the hypothesis of Lemma \ref{lemma:main:linindep:general} implies
\begin{equation}\label{eq:linind:hypothesis}
   {\rm I}(\theta,\bm\varphi_{d-2}):= \sum_{i=1}^N C_i e^{x_i\cdot\omega(\theta,\bm\varphi_{d-2})}\int_{B_{r_i}(0)} e^{ x\cdot (\omega(\theta,\bm\varphi_{d-2})+\hat{\mathrm{e}}_1)}\,{\rm d}x=0,\quad\forall (\theta,\bm\varphi_{d-2})\in\mathbb{R}^{d-1}.
\end{equation}
With $\beta(x)=\sum_{i=1}^N C_i e^{(x-x_i)\cdot\hat{\mathrm{e}}_1} \mathds{1}_{B_{r_i}(x_i)}(x)$,  Lemma \ref{lemma:exponential:in:integralkernel} gives the holomorphicity of ${\rm I}(\theta,\bm\varphi_{d-2})$ with respect to $\theta\in \mathbb{C}$ for each fixed $\bm \varphi_{d-2}\in\mathbb{R}^{d-2}$.
By the unique continuation property of holomorphic functions, \eqref{eq:linind:hypothesis} holds for all $(\theta,\bm \varphi_{d-2})\in\mathbb{C}\times\mathbb{R}^{d-2}$.
By Lemma \ref{lemma:Bessel:complexversion}(b), for all $(\theta,\bm \varphi_{d-2})\in \widetilde K\times\mathbb{R}^{d-2}$,
\begin{equation}\label{eq:3D:linind:diffrad}
\sum_{i=1}^N C_ie^{x_i\cdot\omega(\theta,\bm\varphi_{d-2})}(2\pi r_i)^{\frac{d}{2}} f(\theta)^{-\frac{d}{2}}e^{-\frac{d}{4}\pi{\rm i}}J_{\frac{d}{2}}({\rm i}r_i f(\theta))=0,
\end{equation}
with $f(\theta)=2\cos\tfrac{\theta}{2}$ for all $\theta\in \widetilde K$.
Let
\begin{equation}\label{eq:def:S}
    {S}:=\{v\in\mathbb{S}^{d-1}\subset\mathbb{R}^d\,:\,(x_i-x_j)\cdot v\ne0\mbox{ whenever }x_i\ne x_j\mbox{ with }i,j\in[N]\}.
\end{equation}
By Lemma \ref{lemma:sphere:and:hyperplane}, the set $S$ is open and dense in $\mathbb{S}^{d-1}$. Thus, there exists some $(\widetilde{\theta},\widetilde{\bm\varphi}_{d-2})\in[0,\frac\pi2)\times[0,\pi]^{d-3}\times[0,2\pi]$ satisfying
$\widetilde{\omega}:=\omega(\widetilde{\theta},\widetilde{\bm\varphi}_{d-2})\in {S}$.
Fix any such $(\widetilde{\theta},\widetilde{\bm\varphi}_{d-2})$ and $\widetilde{\omega}$.
Let $\widetilde{\theta}_R:=\widetilde{\theta}-{\rm i}R$ for all $R>0$. Then, clearly, $\widetilde{\theta}_R\in \widetilde K$, and the identity \eqref{eq:3D:linind:diffrad} holds with $(\theta,\bm\varphi_{d-2})=(\widetilde{\theta}_R,\widetilde{\bm\varphi}_{d-2})$ for all $R>0$.
Next we conclude $C_i=0$ for all $i\in[N]$ from the identity \eqref{eq:3D:linind:diffrad} on the half-line $\theta\in\{\widetilde{\theta}_R\,:\,R>0\}$.
We have the following asymptotics as $R\to\infty$:
\begin{align}
\cos\widetilde{\theta}_R&=\tfrac{1}{2}e^{R+{\rm i}\widetilde{\theta}}+O(e^{-R}),\quad\sin\widetilde{\theta}_R=\tfrac{1}{2{\rm i}}e^{R+{\rm i}\widetilde{\theta}}+O(e^{-R}),\label{eq:3D:asymptotic:bigR:cossin}\\
f(\widetilde{\theta}_R)&=2\cos\tfrac{1}{2}\widetilde{\theta}_R=e^{\frac{1}{2}{\rm i}\widetilde{\theta}_R}+e^{-\frac{1}{2}{\rm i}\widetilde{\theta}_R}=e^{\frac{1}{2}R+\frac{1}{2}{\rm i}\widetilde{\theta}}(1+o(1)).
\label{eq:3D:asymptotic:bigR:f}
\end{align}
From \eqref{eq:3D:asymptotic:bigR:cossin}, we obtain for $R\to\infty$, \begin{align*}
\Re(\cos\widetilde{\theta}_R)=\tfrac{1}{2}e^{R}\cos\widetilde{\theta}+O(e^{-R})\quad \mbox{and}\quad
\Re(\sin\widetilde{\theta}_R)=\tfrac{1}{2}e^{R}\sin\widetilde{\theta}+O(e^{-R}).
\end{align*}
Note that with $\widetilde{\bm\varphi}_{d-2}=(\varphi_1,\dots,\varphi_{d-2})$,
\begin{equation*}
{x_i\cdot\omega(\widetilde{\theta}_R,\widetilde{\bm\varphi}_{d-2})}=(x_i\cdot\mathrm{e}_1)\Re(\cos\widetilde{\theta}_R) + \left[\sum_{j=2}^{d-1}(x_i\cdot\mathrm{e}_j)\left(\prod_{k=1}^{j-2}\sin\varphi_k\right)\cos\varphi_{j-1}+(x_i\cdot\mathrm{e}_d)\prod_{k=1}^{d-2}\sin\varphi_k\right]\Re(\sin\widetilde{\theta}_R).
\end{equation*}
Consequently,
\begin{align}
    |e^{x_i\cdot\omega(\widetilde{\theta}_R,\widetilde{\bm\varphi}_{d-2})}| =e^{\frac{1}{2}e^R(x_i\cdot\widetilde{\omega}) + O(e^{-R})}=e^{\frac{1}{2}e^R(x_i\cdot\widetilde{\omega})}(1+o(1)),\quad\mbox{as }R\to\infty.\label{eq:3D:asymptotic:bigR}
\end{align}
By the asymptotic \eqref{eq:3D:asymptotic:bigR:f} and Lemma \ref{lemma:Bessel:asymptotic} for $\delta<\arg (z)<\pi-\delta$ (with $\delta\in(0,\frac{\pi}{4})$), we obtain
\begin{align}
    |J_{\frac{d}{2}}({\rm i}r_i f(\widetilde{\theta}_R))|&=\left|(2\pi {\rm i}r_i f(\widetilde{\theta}_R))^{-\frac{1}{2}}e^{-{\rm i}({\rm i}r_i f(\widetilde{\theta}_R)-\frac{d+1}{4}\pi)}\right|(1+o(1))\nonumber\\
    &=(2\pi r_i e^{\frac{R}{2}})^{-\frac{1}{2}}e^{r_i e^{\frac{R}{2}}}(1+o(1)),\quad\mbox{as }R\to\infty.\label{eq:3D:asymptotics:bigR:J}
\end{align}
Now we prove $C_i=0$ for all $i$.
We rearrange the pairs $\{(r_i,x_i)\}_{i=1}^N$ so that
\begin{equation}\label{eq:ordering:xdotw}
    x_1\cdot\widetilde{\omega}\le x_2\cdot\widetilde{\omega}\le\cdots\le x_{N-1}\cdot\widetilde{\omega}\le x_N\cdot\widetilde{\omega}.
\end{equation}
Since the set $\{(r_i,x_i)\}_{i=1}^N$ consists of distinct pairs and $\widetilde{\omega}\in{S}$, there holds $r_i\ne r_{i+1}$ if $x_i\cdot\widetilde{\omega}= x_{i+1}\cdot\widetilde{\omega}$.
We rearrange the pairs so that $r_i< r_{i+1}$ whenever $x_i\cdot\widetilde{\omega}= x_{i+1}\cdot\widetilde{\omega}$.
Let
$$F_i(R):=e^{x_i\cdot\omega(\widetilde{\theta}_R,\widetilde{\bm\varphi}_{d-2})}( r_i f(\widetilde{\theta}_R)^{-1})^{\frac{d}{2}} J_{\frac{d}{2}}({\rm i}r_i f(\widetilde{\theta}_R)),\quad \forall i\in[N].$$
Since the identity \eqref{eq:3D:linind:diffrad} holds for all $(\theta,\bm\varphi_{d-1})\in \widetilde{K}\times[0,\pi]^{d-3}\times[0,2\pi]$ and $\widetilde{\theta}_R\in \widetilde{K}$ for all $R>0$, which gives $\sum_{i=1}^N C_iF_i(R)=0$ for all $R>0$.
Also, from the relations \eqref{eq:3D:asymptotic:bigR:f}, \eqref{eq:3D:asymptotic:bigR} and \eqref{eq:3D:asymptotics:bigR:J}, we deduce
$$|F_i(R)|=e^{\tfrac{1}{2}e^R(x_i\cdot\widetilde{\omega})}r_i^{\frac{d}{2}}e^{-\tfrac{d}{4}R}(2\pi r_i e^{\frac{R}{2}})^{-\frac{1}{2}}e^{r_i e^{\frac{R}{2}}}(1+o(1)),\quad\mbox{as }R\to\infty.$$
In view of \eqref{eq:ordering:xdotw}, we have $|F_i(R)|=o(|F_{i+1}(R)|)$ as $R\to\infty$ for all $1\le i<N$.
Thus, we recursively obtain
$$C_{k} = \lim_{R\to\infty}\frac{1}{F_k(R)}\sum_{i=1}^N C_iF_i(R) = 0,$$
for $k=N,\dots,1$.
This completes the proof of the lemma.
\end{proof}

\section{Polygonal inclusions}
\label{section:polygons}

In this section we prove the unique identification of polygonal inclusions and their amplitudes.
With Assumption \ref{ass:alpha:simplices}, we prove the uniqueness in Theorems \ref{theorem:simplex}, \ref{theorem:2D:convex} and \ref{theorem:cubes}.
Without Assumption \ref{ass:alpha:simplices}, we adopt a different constraint in Assumption \ref{ass:3dorhigher:multiplepolygons} and prove the uniqueness in Theorem \ref{theorem:cones:manysides}.

\subsection{Proof of Theorem \ref{theorem:simplex}}
\label{sect:theorem:simplex}
The proof of Theorem \ref{theorem:simplex} requires three technical lemmas. The next lemma allows choosing a suitable range of $\omega$ in Lemma \ref{lemma:uniqueness:main:identity}.
\begin{lemma}\label{lemma:convexpolygon}
Let $P$ be a convex polygon.
Fix a vertex $x_0$ of $P$ arbitrarily.
Then there exists a nonempty relatively open subset $O$ of $\mathbb{S}^{d-1}$ such that
\begin{equation}\label{eq:separatingHyperplane}
x_0\cdot\omega>x\cdot\omega, \quad\forall x\in\overline{P}\backslash\{x_0\},\forall \omega\in O.
\end{equation}
\end{lemma}
\begin{proof}
Let $\{z_i\}_{i=1}^n$ be the set of vertices of $P$ with $z_1=x_0$. Then $\{z_i\}_{i=1}^n$ is an irreducible set  satisfying    $$\overline{P}=\left\{\sum_{i=1}^n c_iz_i\,:\,\sum_{i=1}^n c_i=1\mbox{ and }c_i\ge0\,\forall i\in[n]\right\}.$$
    We define a convex and compact set $\widetilde{P}$ as
    $$\widetilde{P}:=\left\{\sum_{i=2}^n c_iz_i\,:\,\sum_{i=2}^n c_i=1\mbox{ and }c_i\ge0\,\forall i\in\{2,\dots,n\}\right\}.$$
By the irreducibility of the set $\{z_i\}_{i=1}^n$, we have $\widetilde {P}\cap\{x_0\}=\emptyset$.
    By the hyperplane separation theorem \cite[Part III, Section 11, Theorem 11.4]{Rockafellar:1970:ConvexAnalysis}, there exists a hyperplane, say $\{x\in\mathbb{R}^d\,:\,x\cdot a = b\}$, that strongly separates the disjoint compact convex sets $\widetilde{P}$ and $\{x_0\}$.
    We may assume $|a|=1$ and $x_0\cdot a>b$. Then $x\cdot a<b$ for all $x\in \widetilde{P}$, by the strong separation.
    Let
    $\varepsilon_0:=\inf\{b-x\cdot a\,:\,x\in \widetilde{P}\}.$ Since $\widetilde{P}$ is compact, we have $\varepsilon_0>0$.
    Fix any $\varepsilon$ satisfying $0<\varepsilon<\varepsilon_0/\max_{x\in \widetilde{P}}|x|$ and set $O:=B_\varepsilon(a)\cap\mathbb{S}^{d-1}$.
    Then, for all $\omega\in O$ and $x\in\widetilde{P}$, there holds
    $$x\cdot\omega-b = x\cdot(\omega-a) + x\cdot a-b \le \varepsilon\max_{x\in \widetilde{P}}|x| - \varepsilon_0 < \varepsilon_0-\varepsilon_0=0.$$
So $O$ satisfies the condition \eqref{eq:separatingHyperplane}.
\end{proof}

The next lemma gives an explicit formula of the integral $\int_{ T_{x_0}(V)}e^{x\cdot y}\d x.$
\begin{lemma}\label{lemma:simplex:integral}
    Let $d\ge2$, $V\in\mathbb{R}^{d\times d}$, $x_0\in\mathbb{R}^d$, $y\in\mathbb{R}^{d}$ and $Y:=V^\top y\in\mathbb{R}^d$. If $\det V>0$, $Y_i\ne0$ and $Y_i\ne Y_j$ for all $i\ne j\in[d]$, then we have
    \begin{equation}\label{eq:lemma:simplex:integral}
        \int_{T_{x_0}(V)}e^{x\cdot y}\,{\rm d}x = e^{x_0\cdot y}|\det V|\left(\sum_{i=1}^d \frac{e^{Y_i}}{Y_i\prod_{j\ne i}(Y_i-Y_j)} + (-1)^d\prod_{i=1}^d\frac{1}{Y_i}\right).
    \end{equation}
\end{lemma}
\begin{proof}
By changing variables $\widetilde{x}=V^{-1}(x-x_0)$, we obtain
\begin{equation}\label{eq:simplex:integral:1}
        \int_{T_{x_0}(V)}e^{x\cdot y}\,{\rm d}x = \int_{T}e^{(x_0+V\widetilde{x})\cdot y}|\det V|\,{\rm d}\widetilde{x}=e^{x_0\cdot y}|\det V|\int_{T}e^{V\widetilde{x}\cdot y}\,{\rm d}\widetilde{x}.
    \end{equation}
    Let $Y:=V^\top y$. The integral over the simplex $T$ can be explicitly computed via
    \begin{align*}
        \int_T e^{V\widetilde{x}\cdot y}\,{\rm d}\widetilde{x}=\int_T e^{Y\cdot\widetilde{x}}\,{\rm d}\widetilde{x}=\int_0^1\cdots\int_0^{1-\sum_{i=1}^{d-1}\widetilde{x}_{i}} e^{Y_d\widetilde{x}_d}\,{\rm d}\widetilde{x}_{d}\cdots e^{Y_1\widetilde{x}_{1}}\,{\rm d}\widetilde{x}_{1}.
    \end{align*}
    If $Y_i\ne0$ and $Y_i\ne Y_j$ for all $i\ne j$, by mathematical induction on $d$, we can prove
\begin{equation}\label{eq:simplex:integral:2}
        \int_T e^{Y\cdot\widetilde{x}}\,{\rm d}\widetilde{x}=\sum_{i=1}^d \frac{e^{Y_i}}{Y_i\prod_{j\ne i}(Y_i-Y_j)} + (-1)^d\prod_{i=1}^d\frac{1}{Y_i}.
    \end{equation}
    The induction step from $d=k$ to $d=k+1$ ($k\ge2$) in the proof of \eqref{eq:simplex:integral:2} follows from the relation
    \begin{align*}
        \int_{T(d=k+1)}e^{Y\cdot\widetilde{x}}\,{\rm d}\widetilde{x} &= \int_0^1 \int_{(1-\widetilde{x}_{1})T(d=k)}e^{Y'\cdot x'}\,{\rm d}x'\,e^{Y_1\widetilde{x}_{1}}\,{\rm d}\widetilde{x}_1\\
        &=\int_0^1 \int_{T(d=k)}e^{(1-\widetilde{x}_{1})(Y'\cdot \widetilde{x}')}\,{\rm d}\widetilde{x}'\,(1-\widetilde{x}_{1})^k e^{Y_1\widetilde{x}_{1}}\,{\rm d}\widetilde{x}_1,
    \end{align*}
    where $Y':=(Y_2,Y_3,\dots,Y_{k+1})$, and we denote the simplex \eqref{eq:T:definition} as $T(d=k)$ and $T(d=k+1)$  to indicate the dimension $d$.
Combining \eqref{eq:simplex:integral:1} and \eqref{eq:simplex:integral:2} gives the desired identity \eqref{eq:lemma:simplex:integral}.
\end{proof}

\begin{remark}
The technical assumptions $Y_i\neq Y_j$ and $Y_i\neq 0$ in the lemma are adopted for the derivation of the formula \eqref{eq:lemma:simplex:integral}.
In the cases $Y_i=Y_j$ or $Y_i=0$, we can still derive  explicit formulas.
For example, if $d=2$, $Y_1=0$ and $Y_2\ne0$, we have
$$\int_{T_{x_0}(V)}e^{x\cdot y}\,{\rm d}x = e^{x_0\cdot y}|\det V|\frac{e^{Y_2}-1}{Y_2^2},$$
whereas if $d=2$ and $Y_1=Y_2\ne0$, we have
$$\int_{T_{x_0}(V)}e^{x\cdot y}\,{\rm d}x = e^{x_0\cdot y}|\det V|\frac{1}{Y_1^2}.$$
\end{remark}

The next lemma extends the  formula \eqref{eq:lemma:simplex:integral} in Lemma \ref{lemma:simplex:integral} to complex variables.
\begin{lemma}\label{lemma:simplex:integral:complex}
Let $d\ge2$, $N\in\mathbb{N}$, $V_i\in\mathbb{R}^{d\times d}$ and $x_i\in\mathbb{R}^d$ such that $\det V_i>0$ for all $i\in[N]$.
Let $\{\mathrm{e}_j\}_{j=1}^d$ be an orthonormal basis of $\mathbb{R}^d$ and let $\omega:\mathbb{C}^{d-1}\to\mathbb{C}^d$ be defined  in \eqref{eqn:omega}.
Then, there exists an open and dense subset $\widetilde{S}$ of $\mathbb{S}^{d-1}$ such that for every $(\widetilde{\theta},\widetilde{\bm\varphi}_{d-2})\in\mathbb{R}^{d-1}$ satisfying
$\omega(\widetilde{\theta},\widetilde{\bm\varphi}_{d-2})\in \widetilde{S},$ there exist $\theta_1$ and $\theta_2$ such that $\theta_1<\widetilde{\theta}<\theta_2$, and for all $\theta\in(\theta_1,\theta_2)+{\rm i}\mathbb{R}$ and $i\in[N]$, there holds
\begin{align}
\int_{T_{x_i}(V_i)}e^{x\cdot (\omega+\mathrm{e}_1)}\,{\rm d}x
=
e^{x_i\cdot (\omega+\mathrm{e}_1)}|\det V_i|\left(\sum_{j=1}^d \frac{e^{v_{ij}(\omega)}}{v_{ij}(\omega)\prod_{k\ne j}(v_{ij}(\omega)-v_{ik}(\omega))} + (-1)^d\prod_{j=1}^d\frac{1}{v_{ij}(\omega)}\right),\label{eq:lemma:simplex:integral:complex}
\end{align}
with $\omega=\omega(\theta,\bm\varphi_{d-2})$ and $v_{ij}(\omega):=(V_i^\top (\omega+\mathrm{e}_1))\cdot\mathrm{e}_j$ for all $j\in[d]$.
\end{lemma}
\begin{proof}
Let $h:\mathbb{S}^{d-1}\to\mathbb{S}^{d-1}$,  $h(\omega(\theta,\bm\varphi_{d-2}))=\omega(\theta+\tfrac{\pi}{2},\bm\varphi_{d-2})$ for all $(\theta,\bm \varphi_{d-2})\in\mathbb{R}^{d-1}$. Let
\begin{equation}\label{eq:S:tilde}
\begin{aligned}
    S_1&:=\{\omega\in\mathbb{S}^{d-1}\,:\,v_{ij}(\omega)\ne0\mbox{ and }v_{ij}(\omega)-v_{ik}(\omega)\ne0,\quad\forall i\in[N],\,\forall j\ne k\in[d]\},\\
    S_2&:=\{\omega\in\mathbb{S}^{d-1}\,:\,(V_i^\top h(\omega))\cdot\mathrm{e}_j\ne0\mbox{ and }(V_i^\top h(\omega))\cdot(\mathrm{e}_j-\mathrm{e}_k)\ne0,\quad\forall i\in[N]\,\forall j\ne k\in[d]\}.
\end{aligned}
\end{equation}
Let $\widetilde{S}=S_1\cap S_2$.
Since $\det V_i\ne0$ for all $i\in[N]$, by Lemma \ref{lemma:sphere:and:hyperplane}, $\widetilde{S}$ is open and dense in $\mathbb{S}^{d-1}$.
Fix any tuple $(\widetilde{\theta},\widetilde{\bm\varphi}_{d-2})\in\mathbb{R}^{d-1}$ such that $\widetilde{\omega}:=\omega(\widetilde{\theta},\widetilde{\bm\varphi}_{d-2})\in \widetilde{S}.$
Since $\widetilde{S}$ is open, there exists an open neighborhood $\mathcal{N}$ of $\widetilde{\omega}$ such that $\mathcal{N}\subset\widetilde{S}$.
Fix any $\theta_1$ and $\theta_2$ that satisfy $\theta_1<\widetilde{\theta}<\theta_2$ and
$$\omega({\theta},\widetilde{\bm\varphi}_{d-2})\in\mathcal{N},\quad\forall\theta\in(\theta_1,\theta_2).$$
We claim that the denominators in \eqref{eq:lemma:simplex:integral:complex} with $\omega=\omega(\theta-{\rm i}R,\widetilde{\bm\varphi}_{d-2})$ for any $\theta\in(\theta_1,\theta_2)$ and $R\in\mathbb{R}$ do not vanish. Fix $\theta\in(\theta_1,\theta_2)$ and $R\in\mathbb{R}$ arbitrarily. If $R=0$, the assertion follows from $\omega(\theta,\widetilde{\bm\varphi}_{d-2})\in\mathcal{N}\subset S_1$.
Suppose $R\ne0$.
From the identities
$\cos(\theta-{\rm i}R)=\cosh R\cos\theta - {\rm i}\sinh R\sin\theta$ and $\sin(\theta-{\rm i}R)=\cosh R\sin\theta + {\rm i}\sinh R\cos\theta,$
we obtain
\begin{align*}
    \Im(\omega(\theta-{\rm i}R,\widetilde{\bm\varphi}_{d-2})) &= (\sinh R) h(\omega(\theta,\widetilde{\bm\varphi}_{d-2})).
\end{align*}
Let
\begin{equation} \label{eqn:tilde-omega}
\widetilde{\omega}_R:=\omega(\theta-{\rm i}R,\widetilde{\bm\varphi}_{d-2}).
\end{equation}
If $v_{ij}(\widetilde{\omega}_R)=0$ for some $i$ and $j$, then
$$(V_i^\top h(\widetilde{\omega}_0))\cdot\mathrm{e}_j = (\sinh R)^{-1}\Im(v_{ij}(\widetilde{\omega}_R)) =0,$$
which contradicts the choice $\widetilde{\omega}_0\in S_2$.
Similarly, if $v_{ij}(\widetilde{\omega}_R)-v_{ik}(\widetilde{\omega}_R)=0$ for some $i$, $j$ and $k$, then
$$(V_i^\top h(\widetilde{\omega}_0))\cdot(\mathrm{e}_j-\mathrm{e}_k) = (\sinh R)^{-1}\Im(v_{ij}(\widetilde{\omega}_R)-v_{ik}(\widetilde{\omega}_R)) =0,$$
which again contradicts the choice $\widetilde{\omega}_0\in S_2$.
Thus the claim holds for all $R\in\mathbb{R}$.
In sum, the right-hand side of \eqref{eq:lemma:simplex:integral:complex} with $\bm\varphi_{d-2}=\widetilde{\bm\varphi}_{d-2}$ is a fraction of holomorphic functions with non-vanishing denominators, and so it is holomorphic in the strip $\theta\in(\theta_1,\theta_2)+{\rm i}\mathbb{R}$.
By Lemma \ref{lemma:exponential:in:integralkernel}, $\int_{T_{x_i}(V_i)}e^{x\cdot (\omega+\mathrm{e}_1)}\,{\rm d}x$ is also holomorphic in $\theta\in\mathbb{C}$.
By Lemma \ref{lemma:simplex:integral}, the identity \eqref{eq:lemma:simplex:integral:complex} holds for all $(\theta,\bm \varphi_{d-2})\in\mathbb{R}^{d-1}$.
Thus by the unique continuation property of holomorphic functions, \eqref{eq:lemma:simplex:integral:complex} holds for all $\theta\in(\theta_1,\theta_2)+{\rm i}\mathbb{R}$ and $\bm \varphi_{d-2}=\widetilde{\bm\varphi}_{d-2}$.
\end{proof}

\begin{lemma}\label{lemma:simplex:linind}
    Fix any $d\ge2$ and $\omega_0\in\mathbb{S}^{d-1}$.
    Let $\{V_i,x_i\}_{i=1}^N\subset\mathbb{R}^{d\times d}\times\mathbb{R}^d$ be a finite set satisfying $T_{x_i}(V_i)\cap T_{x_j}(V_j)=\emptyset$ whenever $i\ne j$.
Suppose that $\{C_i\}_{i=1}^N\subset \mathbb{R}$ satisfies
\begin{equation}\label{eq:linindep:simplices}
    \sum_{i=1}^N C_i \int_{T_{x_i}(V_i)}e^{x\cdot(\omega+\omega_0)}\, {\rm d}x = 0,\quad\forall \omega\in\mathbb{S}^{d-1}.
\end{equation}
Let $P$ be the convex hull of $\bigcup_{i=1}^N T_{x_i}(V_i)$.
If there exists a vertex of $T_{x_1}(V_1)$ that is also a vertex of $P$ but  not a vertex of any simplex in $\{T_{x_i}(V_i)\}_{i=2}^N$, then we have $C_1=0$.
\end{lemma}
\begin{proof}
The proof relies on Lemma \ref{lemma:simplex:integral:complex}. Set $\hat{\mathrm{e}}_1:=\omega_0$ and choose any $\hat{\mathrm{e}}_2,\dots,\hat{\mathrm{e}}_{d}\in\mathbb{S}^{d-1}$ so that $\{\hat{\mathrm{e}}_i\}_{i=1}^{d}$ is an orthonormal basis of $\mathbb{R}^d$.
Set $\omega:\mathbb{C}^{d-1}\to\mathbb{C}^d$ as \eqref{eqn:omega} with $\mathrm{e}_i=\hat{\mathrm{e}}_i$ for all $i\in[N]$.
Define $v_{ij}(\omega)$ with $\omega\in\mathbb{C}^{d-2}$ as in Lemma \ref{lemma:simplex:integral:complex} and set
\begin{equation}\label{eq:I:definition}
    I_i(\omega):=e^{x_i\cdot (\omega+\hat{\mathrm{e}}_1)}|\det V_i|\left(\sum_{j=1}^d \frac{e^{v_{ij}(\omega)}}{v_{ij}(\omega)\prod_{k\ne j}(v_{ij}(\omega)-v_{ik}(\omega))} + (-1)^d\prod_{j=1}^d\frac{1}{v_{ij}(\omega)}\right),
\end{equation}
which is well-defined whenever $\omega\in\mathbb{C}^d$ and the denominators of \eqref{eq:I:definition} do not vanish.
By Lemma \ref{lemma:simplex:integral:complex}, there exists an open and dense set $\widetilde{S}$ of $\mathbb S^{d-1}$ such that for every $(\widetilde{\theta},\widetilde{\bm\varphi}_{d-2})\in\mathbb{R}^{d-1}$ satisfying $\omega(\widetilde{\theta},\widetilde{\bm\varphi}_{d-2})\in\widetilde{S}$, there exist $\theta_1$ and $\theta_2$ such that $\theta_1<\widetilde{\theta}<\theta_2$ and
\begin{equation}\label{eq:explicit:simplices:integral}
    \int_{T_{x_i}(V_i)} e^{ x\cdot (\omega(\theta,\widetilde{\bm\varphi}_{d-2})+\hat{\mathrm{e}}_1)}\,{\rm d}x = I_i(\omega(\theta,\widetilde{\bm\varphi}_{d-2})),\quad\forall i\in[N],\quad \forall \theta\in (\theta_1,\theta_2)+{\rm i}\mathbb{R}.
\end{equation}
Since $\omega$ maps $\mathbb{R}^{d-1}$ to $\mathbb{S}^{d-1}$, \eqref{eq:linindep:simplices} implies
\begin{equation}\label{eq:linind:simplices:coord}
   {\rm I}(\theta,\bm\varphi_{d-2}):= \sum_{i=1}^N C_i \int_{T_{x_i}(V_i)} e^{ x\cdot (\omega(\theta,\bm\varphi_{d-2})+\hat{\mathrm{e}}_1)}\,{\rm d}x=0,\quad\forall (\theta,\bm\varphi_{d-2})\in\mathbb{R}^{d-1}.
\end{equation}
By Lemma \ref{lemma:exponential:in:integralkernel}, ${\rm I}(\theta,\bm\varphi_{d-2})$ is holomorphic in $\theta\in\mathbb{C}$ for any fixed $\bm\varphi_{d-2}\in\mathbb{R}^{d-2}$.
By the unique continuation property of holomorphic functions, \eqref{eq:linind:simplices:coord} holds for all $(\theta,\bm\varphi_{d-2})\in\mathbb{C}\times\mathbb{R}^{d-2}$.
This and \eqref{eq:explicit:simplices:integral} imply
\begin{equation}\label{eq:linind:simplices:coord:ext}
    \sum_{i=1}^N C_i I_i(\omega(\theta,\widetilde{\bm\varphi}_{d-2}))=0,\quad\forall \theta\in (\theta_1,\theta_2)+{\rm i}\mathbb{R}.
\end{equation}
Fix $(\widetilde{\theta},\widetilde{\bm\varphi}_{d-2})\in\mathbb{R}^{d-1}$.
Now we derive the growth rate of the terms in the identity \eqref{eq:linind:simplices:coord:ext} along the half-line $\{\widetilde{\theta}-{\rm i}R\,:\,R\ge0\}$. Let $\widetilde{\omega}_R:=\omega(\widetilde{\theta}-{\rm i}R,\widetilde{\bm\varphi}_{d-2})$ for all $R\ge 0$ and $\widetilde{\omega}:=\widetilde{\omega}_0$.
The estimates \eqref{eq:3D:asymptotic:bigR:cossin} and \eqref{eq:3D:asymptotic:bigR}, for every $x\in\mathbb{R}^d$, imply
\begin{align*}
|x\cdot(\widetilde{\omega}_R+\hat{\mathrm{e}}_1)|&=\tfrac{1}{2}e^R|x\cdot(\widetilde{\omega} + {\rm i}h(\widetilde{\omega}))|(1+o(1))\quad\mbox{as }R\to\infty,\\
\big|e^{x\cdot(\widetilde{\omega}_R+\hat{\mathrm{e}}_1)}\big|&=e^{\frac{1}{2}e^R(x\cdot\widetilde{\omega})}(1+o(1))\quad \mbox{as }R\to \infty,
\end{align*}
with $h(\widetilde{\omega})=\omega(\widetilde{\theta}+\frac{\pi}{2},\widetilde{\bm\varphi}_d)$.
By applying these asymptotic relations to $I_i(\omega)$ in \eqref{eq:I:definition} with $\omega=\widetilde{\omega}_R$, we obtain the following estimate for every $i\in[N]$:
\begin{align*}
    &|I_i(\widetilde{\omega}_R)|=e^{\frac{1}{2}e^R(x_i\cdot\widetilde{\omega})}|\det V_i| \left(\tfrac{1}{2}e^R\right)^{-d} J_i(R) (1+o(1))\quad\mbox{as }R\to\infty,
\end{align*}
with
\begin{align*}
J_i(R)&=
    \prod_{j=1}^d\left|(V_i^\top(\widetilde{\omega}+{\rm i}h(\widetilde{\omega})))_j\right|^{-1}\quad\mbox{if }(V_i^\top\widetilde{\omega})_j<0,\quad \forall j\in[d],
\end{align*}
whereas if $m_i\in[d]$ maximizes $(V_i^\top\widetilde{\omega})_{m_i}$ and satisfies $(V_i^\top\widetilde{\omega})_{m_i}>0$, then
\begin{align*}
    J_i(R)&=\left[e^{\frac{1}{2}e^R(V_i^\top \widetilde{\omega})_{m_i}}\right]\left|(V_i^\top(\widetilde{\omega}+{\rm i}h(\widetilde{\omega})))_{m_i}\right|^{-1}\prod_{k\ne {m_i}}\left|(V_i^\top(\widetilde{\omega}+{\rm i}h(\widetilde{\omega})))_{m_i}-(V_i^\top(\widetilde{\omega}+{\rm i}h(\widetilde{\omega})))_k\right|^{-1}.
\end{align*}
Let $x_0$ be a vertex of $T_{x_1}(V_1)$ that is also a vertex of $P$ but not a vertex of any simplex in $\{T_{x_i}(V_i)\}_{i=2}^N$.
Since $P$ is a convex hull of $\{T_{x_i}(V_i)\}_{i=1}^N$ and $\widetilde{S}$ is dense in $\mathbb{S}^{d-1}$, by Lemma \ref{lemma:convexpolygon}, there exists an $\widetilde{\omega}\in\widetilde{S}$ satisfying
$$x_0\cdot \widetilde{\omega} > \max\{x\cdot\widetilde{\omega}\,:\, x\ne x_0\mbox{ and }x\mbox{ is a vertex of }T_{x_i}(V_i)\mbox{ for some }i\in[N]\}.$$
For any such $\widetilde{\omega}$, there exists a tuple $(\widetilde{\theta},\widetilde{\bm\varphi}_{d-2})\in\mathbb{R}^{d-1}$ satisfying $\widetilde{\omega}=\omega(\widetilde{\theta},\widetilde{\bm\varphi}_{d-2})$, since $\omega$ maps $\mathbb{R}^{d-1}$ onto $\mathbb{S}^{d-1}$.
Note that the growth rate for $|I_i(\widetilde{\omega}_R)|$ as $R\to\infty$ is $$e^{\frac{1}{2}e^R\max\{x\cdot\widetilde{\omega}\,:\, x\mbox{ is a vertex of }T_{x_i}(V_i)\}-dR}.$$
Thus $|I_i(\widetilde{\omega}_R)|=o(|I_1(\widetilde{\omega}_R)|)$ for all $i\in\{2,\dots,N\}$ as $R\to\infty$. Therefore, we have
$$C_1=\lim_{R\to\infty}\frac{1}{I_1(\widetilde{\omega}_R)}\sum_{i=1}^N C_i I_i(\widetilde{\omega}_R) = 0,$$
which is the desired conclusion.
\end{proof}

\begin{proof}[Proof of Theorem \ref{theorem:simplex}]
By Lemma \ref{lemma:uniqueness:main:identity}, we have
\begin{equation}\label{eq:simplices:basic:identity}
\int_\Om (\alpha^1(x)-\alpha^2(x)) e^{x\cdot(\omega+\omega_0)}\,{\rm d} x =0,\quad\forall\omega\in\mathbb{S}^{d-1}.
\end{equation}
We prove $\alpha^1\equiv\alpha^2$ by contradiction. Suppose $\alpha^1\not\equiv\alpha^2$.
Let $\{T_{x_i}(V_i)\}_{i=1}^N$ be the set satisfying all the conditions in Assumption \ref{ass:alpha:simplices} for $\alpha^1$ and $\alpha^2$.
Then, $\{T_{x_i}(V_i)\}_{i=1}^N$ is an irreducible set of disjoint simplices such that there exists $\{\alpha_i\}_{i=1}^N$ satisfying
\begin{equation}\label{eq:simplices:piecewiseconstant}
    \alpha^1(x)-\alpha^2(x)=\sum_{i=1}^N\alpha_i\mathds{1}_{T_{x_i}(V_i)}(x),\quad\forall x\in \Om.
\end{equation}
The irreducibility implies that $\alpha_i\ne0$ for all $i\in[N]$.
Combining \eqref{eq:simplices:basic:identity} and \eqref{eq:simplices:piecewiseconstant} gives
$$\sum_{i=1}^{N}\alpha_i\int_{T_{x_i}(V_i)} e^{x\cdot(\omega+\omega_0)}\,{\rm d}x = 0,\quad\forall\omega\in\mathbb{S}^{d-1}.$$
Let $P$ be the convex hull of $\bigcup_{i=1}^N T_{x_i}(V_i)$.
From Assumption \ref{ass:alpha:simplices}, up to rearranging indices for $\{(\alpha_i,T_{x_i}(V_i))\}_{i=1}^N$, there exists a vertex of $T_{x_1}(V_1)$ that is also a vertex of $P$ but is not a vertex of any simplex in $\{T_{x_i}(V_i)\}_{i=2}^N$.
By Lemma \ref{lemma:simplex:linind}, $\alpha_1=0$,  which contradicts the condition $\alpha_i\ne0$ for all $i\in[N]$. Thus we conclude $\alpha^1\equiv\alpha^2$.
\end{proof}

\subsection{Proof of Theorem \ref{theorem:2D:convex}}
\label{subsect:2D:convex}
In this part, using  Theorem \ref{theorem:simplex} and the following two lemmas, we prove Theorem \ref{theorem:2D:convex}.
\begin{lemma}\label{lemma:2D:simplepoly:onesimplex}
Let $P$ be a union of polygonal open sets in $\mathbb{R}^2$.
Let $x_0$ be a vertex of $P$.
Suppose that there exists an $\varepsilon>0$ such that $B_\varepsilon(x_0)\cap P$ is equal to a connected circular sector with center $x_0$.
Then, there exists a set $\{T_{x_i}(V_i)\}_{i=1}^N$ of disjoint simplices $($triangles$)$ such that
\begin{equation}\label{eq:2D:simplex:decomp}
    \bigcup_{i=1}^N \overline{T_{x_i}(V_i)}=\overline{P},\quad x_0\in\overline{T_{x_1}(V_1)}\quad\mbox{and}\quad x_0\not\in\bigcup_{i=2}^N\overline{T_{x_i}(V_i)}.
\end{equation}
\end{lemma}
\begin{proof}
    Let $y_0$ and $z_0$ be the end points of the arc of the circular sector $B_\varepsilon(x_0)\cap P$.
    We decompose $P$ into a family of disjoint simplices, denoted by $\{T_{x_i}(V_i)\}_{i=1}^N$, that includes the triangle $T_{x_1}(V_1)$ defined by the set of vertices being $\{x_0,y_0,z_0\}$.
    Then the desired relation \eqref{eq:2D:simplex:decomp}  follows.
\end{proof}
\begin{remark}
Lemma \ref{lemma:2D:simplepoly:onesimplex} holds for polygons in $\mathbb{R}^2$, but it is false in $\mathbb{R}^d$ for $d\ge3$. One counterexample for polygons in $\mathbb{R}^d$ with $d\ge3$ is given by
    $\{(y_1,y_2,\dots,y_d)\in\mathbb{R}^d\,:\,|y_1|+|y_2|+\cdots+|y_d|<1\},$
    which is a polygon such that every vertex has $2d-2$ edges.
\end{remark}
\begin{lemma}\label{lemma:2D:convex}
Let $d=2$, and let Assumption \ref{ass:alpha:general} with $\alpha=\alpha_{\rm in}$ holds. For each $i\in\{1,2\}$, let $\alpha_{\rm c}^i$ be a constant, let $D^i$ be a convex polygon such that $\overline{D^i}\subset\Om$, and define
$\alpha^i(x) := \alpha_{\rm in}(x) + \alpha_{\rm c}^i\mathds{1}_{D^i}(x).$
Let $\alpha^1$ and $\alpha^2$ satisfy Assumption \ref{ass:alpha:general}.
Then $\alpha^1$ and $\alpha^2$ satisfy Assumption \ref{ass:alpha:simplices}.
\end{lemma}
\begin{proof}
Note that
    \begin{align*}
        &\alpha^1(x)-\alpha^2(x)=\alpha_{\rm c}^1\mathds{1}_{D^1}(x)-\alpha_{\rm c}^2\mathds{1}_{D^2}(x)\\
        =&\alpha_{\rm c}^1\mathds{1}_{D^1\backslash D^2}(x)+(\alpha_{\rm c}^1-\alpha_{\rm c}^2)\mathds{1}_{D^1\cap D^2}(x)-\alpha_{\rm c}^2\mathds{1}_{D^2\backslash D^1}(x),\quad\forall x\in\Om.
    \end{align*}
    Suppose that $\alpha^1\ne\alpha^2$.
    Let $P$ be the convex hull of the support of $\alpha^1-\alpha^2$.
    Below, we prove that Assumption \ref{ass:alpha:simplices} is satisfied for the cases $D^1=D^2$ and $D^1\ne D^2$ separately.

    \smallskip
    \smallskip
\noindent{\bf Case 1 ($D^1=D^2$).} By the hypothesis $\alpha^1\ne\alpha^2$, we have $\alpha_{\rm c}^1\neq \alpha_{\rm c}^2$, and thus, $\alpha^1-\alpha^2=(\alpha_{\rm c}^1-\alpha_{\rm c}^2)\mathds{1}_{D^1}$ has the support $D^1$.
Since $D^1$ is convex, we have $P=D^1$.
    Fix $x_0$ to be any vertex of $P$.
    Let $\varepsilon=\min\{|x_0-x|\,:\,x\mbox{ is a vertex of }P\}$.
    Since $P$ is a convex polygon, $x_0$ and $\varepsilon$ satisfy the assumption of Lemma \ref{lemma:2D:simplepoly:onesimplex}.
    Then Lemma \ref{lemma:2D:simplepoly:onesimplex} implies that Assumption \ref{ass:alpha:simplices} holds.

    \smallskip
    \smallskip
    \noindent{\bf Case 2 ($D^1\neq D^2$)}
    If $\alpha_{\rm c}^1\ne\alpha_{\rm c}^2$, we have $\mathrm{supp}(\alpha^1-\alpha^2)=D^1\cup D^2$. Since $P$ is the convex hull of $\mathrm{supp}(\alpha^1-\alpha^2)$, every vertex of $P$ is also a vertex of either $D^1$ or $D^2$.
However, since $D^1$ and $D^2$ are both convex polygons, if every vertex of $P$ is also the vertex of both $D^1$ and $D^2$, then the only possible case is $D^1=D^2$, which contradicts the hypothesis. (The pair $(\alpha^1,\alpha^3)$ in Fig. \ref{fig:examples} is one counterexample of the statement without the convexity assumption.) Thus, there exists a vertex $x_0$ of both $P$ and $D^1$ (or respectively $D^2$) such that $\operatorname{dist}(x_0,D^2)>0$ (or respectively $\operatorname{dist}(x_0,D^1)>0$).
This statement is also true for the case $\alpha_{\rm c}^1=\alpha_{\rm c}^2$ when we have $\mathrm{supp}(\alpha^1-\alpha^2)=(D^1\backslash D^2)\cup(D^2\backslash D^1)$. (The pair $(\alpha^1,\alpha^2)$ in Fig. \ref{fig:examples} is one counterexample without the convexity assumption.)
Without loss of generality, let $x_0$ be a vertex of $D^1$ and set
$$\varepsilon = \min\left(\operatorname{dist}(x_0,D^2),\min\{|x_0-x|\,:\,x\mbox{ is a vertex of }D^1\}\right).$$
Then the vertex $x_0$ of $D^1$ and $\varepsilon$ satisfy the assumption for Lemma \ref{lemma:2D:simplepoly:onesimplex}.
Thus, Assumption \ref{ass:alpha:simplices} is satisfied by Lemma \ref{lemma:2D:simplepoly:onesimplex}.
\end{proof}

\begin{proof}[Proof of Theorem \ref{theorem:2D:convex}]
Combining Lemma \ref{lemma:2D:convex} with Theorem \ref{theorem:simplex} gives Theorem \ref{theorem:2D:convex}.
\end{proof}

\subsection{Rectangular inclusions}\label{section:rectangles}

Let $K:=(-\frac12,\frac12)^d\subset\mathbb{R}^d$ and $K_{\delta}(x):=x+\delta\cdot K$ for all $\delta\in(0,\infty)^d$ and $x\in\mathbb{R}^d$.
\begin{ass}\label{ass:alpha:cubes}
$\alpha^i:\Om\to(0,1)$ is measurable and satisfies $\overline{\alpha^i}<2\underline{\alpha^i}$ with $i=1,2$. Also, there exist some $N\in\mathbb{N}$, $\{\delta_j\}_{j=1}^N\subset\mathbb{R}_+^d$, $\{x_j\}_{j=1}^N\subset\Om$ and $\{\alpha_j\}_{j=1}^N\subset(-1,1)$ such that $\{K_{\delta_j}(x_j)\}_{j=1}^N$ is a family of pairwise disjoint sets satisfying $\overline{K_{\delta_j}(x_j)}\subset\Om$ for all $j\in[N]$ and
\begin{equation}\label{eq:rect:assumption}
    \alpha^1(x)-\alpha^2(x)=\sum_{j=1}^N \alpha_j\mathds{1}_{K_{\delta_j}(x_j)}(x),\quad\forall x\in\Om.
\end{equation}
\end{ass}
\begin{remark}
Note that Assumption \ref{ass:alpha:cubes} does not require $\{\overline{K_{\delta_j}(x_j)}\}_{j=1}^N$ to be pairwise disjoint. So it can be applied to rectangular partitions of domains.
\end{remark}

The next result gives the unique identifiability of rectangular inclusions and their amplitudes.
\begin{theorem}
\label{theorem:cubes}
Let Assumption \ref{ass:g:exptype} hold. Let $U=U^i$ be the solution to  \eqref{eq:2D3D:ibvp} with $\alpha=\alpha^i$ for each $i=1,2$ satisfying Assumption \ref{ass:alpha:cubes}. If $\p_\nu U^1(t,x)=\p_\nu U^2(t,x)$ for all $(t,x)\in I\times\p\Om$,
then we have $\alpha^1=\alpha^2$.
\end{theorem}
\begin{proof}
We prove that Assumption \ref{ass:alpha:cubes} implies Assumption \ref{ass:alpha:simplices}.
Then, from Theorem \ref{theorem:simplex}, Theorem \ref{theorem:cubes} follows.
Suppose that $\alpha^1$ and $\alpha^2$ satisfy Assumption \ref{ass:alpha:cubes}.
If $\alpha^1=\alpha^2$ almost everywhere, Assumption \ref{ass:alpha:simplices} holds trivially. Otherwise, \eqref{eq:rect:assumption} holds for some $N\in\mathbb{N}$, $\{\delta_j\}_{j=1}^N\subset\mathbb{R}_+^d$, $\{x_j\}_{j=1}^N\subset\Om$ and $\{\alpha_j\}_{j=1}^N\subset(-1,1)\backslash\{0\}$.
Let $P$ be the convex hull of $\bigcup_{j=1}^N K_{\delta_j}(x)$ and $y$ be any vertex of $P$.
Then $y$ is also a vertex of $K_{\delta_k}(x_k)$ for some $k\in[N]$.
However, $y$ cannot be a vertex of $K_{\delta_\ell}(x_\ell)$ for $\ell\ne k$: Otherwise, $y$ cannot be a vertex of $P$ since $K_{\delta_\ell}(x_\ell)$ are rectangles with parallel sides.
Thus, by subdividing the rectangles into simplices including the simplex determined by the convex combination of $y$ and the $d$ vertices of $K_{\delta_k}(x_k)$ adjacent to $y$, we derive Assumption \ref{ass:alpha:simplices}.
\end{proof}

\subsection{Multiple polygonal inclusions}\label{subsection:various:cases}

We arbitrarily fix $\alpha_{\rm in}$ satisfying Assumption \ref{ass:alpha:general}.
\begin{ass}\label{ass:alpha:poly:nocommonvertex}
For $i=1,2$, $\alpha^i$ satisfies Assumption \ref{ass:alpha:general}, and there exist $M^i\in\mathbb{N}$, $\{\alpha_j^{i}\}_{j=1}^{M^i}\subset (-1,1)$ and some polyhedrons $\{P_j^i\}_{j=1}^{M^i}$ such that
    $$\alpha^i(x) = \alpha_{\rm in}(x) +\sum_{j=1}^{M^i}\alpha_j^i\mathds{1}_{P_j^i}(x),\quad\forall x\in\Om,$$
    where $\overline{P_j^i}\subset\Om$, each surface $\p P_j^i$ is not self intersecting, $\operatorname{dist}(P_j^i,P_k^i)>0$ whenever $j\ne k$, and the set $\mathcal{V}^i$ of vertices of all polyhedrons $\{P_j^i\}_{j=1}^{M^i}$ satisfying $\mathcal{V}^1\cap\mathcal{V}^2=\emptyset$ unless $\alpha^1=\alpha^2$ almost everywhere in $\Om$.
\end{ass}
Let $\alpha^1$ and $\alpha^2$ satisfy Assumption \ref{ass:alpha:poly:nocommonvertex} and suppose that $\alpha^1\ne\alpha^2$.
Then, there exists an irreducible finite set of disjoint simplices $\{T_{x_i}(V_i)\}_{i=1}^N$ in $\Om$ such that $\alpha^1-\alpha^2$ is constant in $T_{x_i}(V_i)$ for each $i\in[N]$, and $\alpha^1=\alpha^2$ in $\Om\backslash\bigcup_{i=1}^N \overline{T_{x_i}(V_i)}$.
Let $P$ be the convex hull of $\mathrm{supp}(\alpha^1-\alpha^2)$ and let $x_0$ be a vertex of $P$.
If $d=2$, $x_0$ is the endpoint of exactly two edges of $\mathrm{supp}(\alpha^1-\alpha^2)$, so we can choose $\{T_{x_i}(V_i)\}_{i=1}^N$ so that $x_0$ is a vertex of exactly one of the simplices in $\{T_{x_i}(V_i)\}_{i=1}^N$.
However, if $d\ge3$, $x_0$ can be the endpoint of arbitrarily many edges of $\mathrm{supp}(\alpha^1-\alpha^2)$, so there can be arbitrarily big lower bound of number of simplices in $\{T_{x_i}(V_i)\}_{i=1}^N$ that has the vertex $x_0$.

Assumption \ref{ass:alpha:poly:nocommonvertex} implies Assumption \ref{ass:alpha:simplices} when $d=2$. So Theorem \ref{theorem:simplex} is applicable to the pairs $(\alpha^1,\alpha^2)$ satisfying Assumption \ref{ass:alpha:poly:nocommonvertex} in $\mathbb{R}^2$.
    This is not true if $d\ge3$, which motivates the following assumption.
Like in Section \ref{sect:theorem:simplex}, for each triangulation $\{T_{x_i}(V_i)\}_{i=1}^N$, we define $\widetilde{S}$ as in \eqref{eq:S:tilde}.

\begin{ass}\label{ass:3dorhigher:multiplepolygons}
    Let $d\ge3$.
Let $\alpha^1$ and $\alpha^2$ satisfy Assumption \ref{ass:alpha:poly:nocommonvertex}.
Unless $\alpha^1=\alpha^2$ in $\Om$, there exist a vertex $z$ of the convex hull $P$ of  $\mathrm{supp}(\alpha^1-\alpha^2)$, a triangulation $\{T_{x_i}(V_i)\}_{i=1}^{N}$ of the support of $\alpha^1-\alpha^2$ among which only those in $\{T_{x_i}(V_i)\}_{i=1}^{m}$ have $z=x_i$ as a vertex, and a vector $\widetilde{\omega}\in\widetilde{S}$ such that the half-space $z+\{x\in\mathbb{R}^d\,:\,x\cdot\widetilde{\omega}<0\}$ contains $P$, and at least one of the following conditions holds:
\begin{itemize}
    \item[\rm(a)] The cone $z+\{x\in\mathbb{R}^d\,:\,x\cdot\widetilde{\omega}<-|x|\cos(\frac{\pi}{2d})\}$ contains $\bigcup_{i=1}^{m} \overline{T_{x_i}(V_i)}$.
    \item[\rm(b)] There exists some $\widetilde{\omega}'\in\mathbb{S}^{d-1}$ such that $\widetilde{\omega}'\cdot\widetilde{\omega}=0$ and
    \begin{equation}\label{eq:3dorhigher:multi}
        \sum_{i=1}^m |\det V_i|\textstyle\prod_{j=1}^d (V_i^\top (\widetilde{\omega}+{\rm i}\widetilde{\omega}'))_j^{-1}\ne0.
    \end{equation}
\end{itemize}
\end{ass}

\begin{lemma}
Assumption \ref{ass:3dorhigher:multiplepolygons}(a) implies Assumption \ref{ass:3dorhigher:multiplepolygons}(b).
\end{lemma}
\begin{proof}
Fix any $\widetilde{\omega}'\in\mathbb{S}^{d-1}$ such that $\widetilde{\omega}'\cdot\widetilde{\omega}=0$.
We prove the following sufficient condition for \eqref{eq:3dorhigher:multi}:
\begin{equation}\label{ineq:nonzero:proof}
    (-1)^d\Re\left(\prod_{j=1}^d (V_i^\top (\widetilde{\omega}+{\rm i}\widetilde{\omega}'))_j^{-1}\right)>0,\quad\forall i\in[m].
\end{equation}
Since the cone $z+\{x\in\mathbb{R}^d\,:\,x\cdot\widetilde{\omega}<-|x|\cos(\frac{\pi}{2d})\}$ contains $\bigcup_{i=1}^{m} \overline{T_{x_i}(V_i)}$, it also contains the vertex $z+V_i\mathrm{e}_j$ of $T_{x_i}(V_i)$ for every $i\in[m]$ and $j\in[d]$.
Thus we have $V_i\mathrm{e}_j \cdot \widetilde{\omega}< -|V_i\mathrm{e}_j|\cos(\frac{\pi}{2d}),$
which implies
\begin{align*}
    |(V_i^\top\widetilde{\omega})_j|^2 &> |V_i\mathrm{e}_j|^2\cos^2\left(\frac{\pi}{2d}\right),\\[2mm]
    |(V_i^\top\widetilde{\omega}')_j|^2 &\le |V_i\mathrm{e}_j - (V_i\mathrm{e}_j \cdot \widetilde{\omega})\widetilde{\omega}|^2 =|V_i\mathrm{e}_j|^2-|V_i\mathrm{e}_j \cdot \widetilde{\omega}|^2\\
    &< |V_i\mathrm{e}_j|^2\left(1-\cos^2\left(\frac{\pi}{2d}\right)\right)=|V_i\mathrm{e}_j|^2\sin^2\left(\frac{\pi}{2d}\right).
\end{align*}
Thus, for every $i\in[m]$ and $j\in[d]$, we have
$$\left|\arg\left(1+{\rm i}(V_i^\top \widetilde{\omega})_j^{-1}(V_i^\top \widetilde{\omega}')_j\right)\right|<\frac{\pi}{2d}.$$
We also have $(V_i^\top\widetilde{w})_j<0$ for all $i,j$ so that
\begin{align*}
    \arg\left(\prod_{j=1}^d (V_i^\top (\widetilde{\omega}+{\rm i}\widetilde{\omega}'))_j^{-1}\right)&=\arg\left(\prod_{j=1}^d\frac{(V_i^\top \widetilde{\omega})_j-{\rm i}(V_i^\top\widetilde{\omega}')_j}{(V_i^\top\widetilde{\omega})_j^2+(V_i^\top\widetilde{\omega}')_j^2}\right)=d\pi+\sum_{j=1}^d \arg(1+{\rm i}(V_i^\top \widetilde{\omega})_j^{-1}(V_i^\top \widetilde{\omega}')_j).
\end{align*}
By combining the preceding relations, we arrive at
$$\left|\arg\left(\prod_{j=1}^d (V_i^\top (\widetilde{\omega}+{\rm i}\widetilde{\omega}'))_j^{-1}\right) - d\pi\right|<\sum_{j=1}^d\frac{\pi}{2d}=\frac{\pi}{2},$$
which gives the desired assertion \eqref{ineq:nonzero:proof}.
\end{proof}
The next result gives the unique identifiability of polygonal inclusions and their amplitudes.
\begin{theorem}
\label{theorem:cones:manysides}
Let Assumptions \ref{ass:alpha:poly:nocommonvertex} and \ref{ass:g:exptype} hold, and also Assumption \ref{ass:3dorhigher:multiplepolygons} if $d\ge3$, and let $U=U^i$ be the solution to  \eqref{eq:2D3D:ibvp} with $\alpha=\alpha^i$, $i=1,2$.
If $\p_\nu U^1(t,x)=\p_\nu U^2(t,x)$ for all $(t,x)\in I\times\p\Om$,
then $\alpha^1=\alpha^2$.
\end{theorem}
\begin{proof}
If $d=2$, Assumption \ref{ass:alpha:poly:nocommonvertex} implies Assumption \ref{ass:alpha:simplices}, so the uniqueness result follows directly from Theorem \ref{theorem:simplex}. Now
suppose that $d\ge3$, $\alpha^1\ne\alpha^2$ and Assumption \ref{ass:3dorhigher:multiplepolygons} is satisfied for the vertex $x_0$, an irreducible triangulation $\{T_{x_i}(V_i)\}_{i=1}^N$, $\widetilde{\omega}\in\widetilde{S}$ and $\widetilde{\omega}'\in\mathbb{S}^{d-1}$.
Then there exists $\{C_i\}_{i=1}^N\subset(-1,1)$ such that
$$\alpha^1(x)-\alpha^2(x) = \sum_{i=1}^N C_i\mathds{1}_{T_{x_i}(V_i)}(x),\quad\forall x\in\Om.$$
We enumerate the triangulation again so that $x_0$ is a vertex of $\{T_{x_i}(V_i)\}_{i=1}^m$ but not of $\{T_{x_i}(V_i)\}_{i=m+1}^{N}$.
Note that $x_i=x_0$ for all $i\in[m]$.
Let $v_{ij}(\omega):=(V_i^\top(\omega+\omega_0))_j$ for all $\omega\in\mathbb{C}^d$, and define $I_i(\omega)$ for $\omega\in\mathbb{C}$ by \eqref{eq:I:definition}, provided $v_{ij}(\omega)\ne0$.
Then, by Lemmas \ref{lemma:uniqueness:main:identity} and \ref{lemma:simplex:integral}, we have
\begin{equation}\label{eq:sigCIeq0}
    \sum_{i=1}^N C_i I_i(\omega)=0,\quad\forall \omega\in\mathbb{S}^{d-1}.
\end{equation}
We also define $\widetilde{\omega}_R$ by \eqref{eqn:tilde-omega}.
Since $x_0+V_i\mathrm{e}_j\in T_{x_i}(V_i)\subset P\subset x_0+\{x\in\mathbb{R}^d\,:\,x\cdot\widetilde{\omega}<0\}$, we have $(V_i^\top\widetilde{\omega})_j=V_i\mathrm{e}_j\cdot\widetilde{\omega}<0$ for all $i,j$. Thus, $v_{ij}(\widetilde{\omega}_R)=\frac{1}{2}e^R (V_i^\top (\widetilde{\omega}+{\rm i}\widetilde{\omega}'))_j (1+o(1))$ has a strictly negative real part for all $i,j$ and $R\gg1$. Therefore, the last term of \eqref{eq:I:definition} dominates as
$$I_i(\widetilde{\omega}_R)=e^{\frac{1}{2}e^R(x_0\cdot(\widetilde{\omega}+{\rm i}\widetilde{\omega}'))-dR}|\det V_i|(-2)^d \prod_{j=1}^d (V_i^\top (\widetilde{\omega}+{\rm i}\widetilde{\omega}'))_j^{-1} (1+o(1))\quad\mbox{as }R\to\infty,$$
for $i\in [N]$. This identity and the holomorphic extension of \eqref{eq:sigCIeq0} with $\omega=\widetilde{\omega}_R$ in the variable $\theta-{\rm i}R\in \{\theta-{\rm i}R\,:\,\theta\in(\theta_1,\theta_2)\mbox{ and }R>0\}$ in Section \ref{sect:theorem:simplex} imply
\begin{align*}
    &\quad \sum_{i=1}^m C_i|\det V_i|\prod_{j=1}^d (V_i^\top (\widetilde{\omega}+{\rm i}\widetilde{\omega}'))_j^{-1}
    =(-2)^{-d}\lim_{R\to\infty}e^{-\frac{1}{2}e^R(x_0\cdot(\widetilde{\omega}+{\rm i}\widetilde{\omega}'))+dR}\sum_{i=1}^N C_i I_i(\widetilde{\omega}_R)=0.
\end{align*}
Under Assumption \ref{ass:alpha:poly:nocommonvertex}, there is no common vertex of polygons, so we have $C_i=C_1$ for all $i\in\{2,\dots,m\}$.
Under Assumption \ref{ass:3dorhigher:multiplepolygons}, this gives $C_1=0$, which contradicts the irreducibility of the triangulation.
\end{proof}

\bibliographystyle{abbrv}
\bibliography{reference2}

\begin{thebibliography}{10}

\bibitem{AIP}
G.~Alessandrini, V.~Isakov, and J.~Powell.
\newblock Local uniqueness in the inverse conductivity problem with one
  measurement.
\newblock {\em Trans. Amer. Math. Soc.}, 347(8):3031--3041, 1995.

\bibitem{ADKL}
H.~Ammari, Y.~Deng, H.~Kang, and H.~Lee.
\newblock Reconstruction of inhomogeneous conductivities via the concept of
  generalized polarization tensors.
\newblock {\em Ann. Inst. Henri Poincar{\'e}, Anal. Non Lin{\'e}aire},
  31(5):877--897, 2014.

\bibitem{AK}
H.~Ammari and H.~Kang.
\newblock {\em {Polarization and Moment Tensors}}.
\newblock Springer, New York, 2007.

\bibitem{Artin:1964:Gamma}
E.~Artin.
\newblock {\em {The Gamma Function}}.
\newblock Winston, New York-Toronto-London, 1964.

\bibitem{ChechkinGorenflo:2005}
A.~V. Chechkin, R.~Gorenflo, and I.~M. Sokolov.
\newblock Fractional diffusion in inhomogeneous media.
\newblock {\em J. Phys. A: Math. Gen.}, 38:L679--L684, 2005.

\bibitem{DLLi}
Y.~Deng, J.~Li, and H.~Liu.
\newblock On identifying magnetized anomalies using geomagnetic monitoring.
\newblock {\em Arch. Ration. Mech. Anal.}, 231(1):153--187, 2019.

\bibitem{DLLi2}
Y.~Deng, J.~Li, and H.~Liu.
\newblock On identifying magnetized anomalies using geomagnetic monitoring
  within a magnetohydrodynamic model.
\newblock {\em Arch. Ration. Mech. Anal.}, 235(1):691--721, 2020.

\bibitem{Ev}
L.~C. Evans.
\newblock {\em {Partial Differential Equations}}.
\newblock AMS, Providence, RI, 2nd edition, 2010.

\bibitem{Fedotov:2012}
S.~Fedotov and S.~Falconer.
\newblock Subdiffusive master equation with space-dependent anomalous exponent
  and structural instability.
\newblock {\em Phys. Rev. E}, 85:031132, 6 pp., 2012.

\bibitem{FrIs}
A.~Friedman and V.~Isakov.
\newblock On the uniqueness in the inverse conductivity problem with one
  measurement.
\newblock {\em Indiana Univ. Math. J.}, 38(3):563--579, 1989.

\bibitem{FV}
A.~Friedman and M.~Vogelius.
\newblock Identification of small inhomogeneities of extreme conductivity by
  boundary measurements: {A} theorem on continuous dependence.
\newblock {\em Arch. Ration. Mech. Anal.}, 105(4):299--326, 1989.

\bibitem{HJK:2025:ISDVO}
J.~Hong, B.~Jin, and Y.~Kian.
\newblock Identification of a spatially-dependent variable order in
  one-dimensional subdiffusion.
\newblock {\em SIAM J. Math. Anal.}, 57(2):1315--1341, 2025.

\bibitem{HJK:2026:2D3D}
J.~Hong, B.~Jin, and Y.~Kian.
\newblock Unique and stable recovery of space-variable order in
  multidimensional subdiffusion.
\newblock {\em SIAM J. Math. Anal.}, 58(1):238--259, 2026.

\bibitem{Ik98}
M.~Ikehata.
\newblock Reconstruction of the shape of the inclusion by boundary
  measurements.
\newblock {\em Commun. Partial Differ. Equations}, 23(7-8):1459--1474, 1998.

\bibitem{Ikk1}
M.~Ikehata.
\newblock Size estimation of inclusion.
\newblock {\em J. Inverse Ill-Posed Probl.}, 6(2):127--140, 1998.

\bibitem{IkehataKian:2023}
M.~Ikehata and Y.~Kian.
\newblock The enclosure method for the detection of variable order in
  fractional diffusion equations.
\newblock {\em Inverse Probl. Imaging}, 17(1):180--202, 2023.

\bibitem{JK}
J.~Janno and N.~Kinash.
\newblock Reconstruction of an order of derivative and a source term in a
  fractional diffusion equation from final measurements.
\newblock {\em Inverse Problems}, 34(2):025007, 19 pp., 2018.

\bibitem{Jin:2021}
B.~Jin.
\newblock {\em {Fractional Differential Equations}}.
\newblock Springer-Nature, Switzerland, 2021.

\bibitem{JiKi2}
B.~Jin and Y.~Kian.
\newblock Recovery of the order of derivation for fractional diffusion
  equations in an unknown medium.
\newblock {\em SIAM J. Appl. Math.}, 82(3):1045--1067, 2022.

\bibitem{JiKi}
B.~Jin and Y.~Kian.
\newblock Recovery of a distributed order fractional derivative in an unknown
  medium.
\newblock {\em Commun. Math. Sci.}, 21(7):1791--1813, 2023.

\bibitem{JinRundell:2015}
B.~Jin and W.~Rundell.
\newblock A tutorial on inverse problems for anomalous diffusion processes.
\newblock {\em Inverse Problems}, 31(3):035003, 40 pp., 2015.

\bibitem{KaltenbacherRundell:2023}
B.~Kaltenbacher and W.~Rundell.
\newblock {\em {Inverse Problems for Fractional Partial Differential
  Equations}}.
\newblock AMS, Providence, RI, 2023.

\bibitem{KaSe}
H.~Kang, J.~K. Seo, and D.~Sheen.
\newblock The inverse conductivity problem with one measurement: stability and
  estimation of size.
\newblock {\em SIAM J. Math. Anal.}, 28(6):1389--1405, 1997.

\bibitem{Kav}
O.~Kavian.
\newblock Lectures on parameter identification.
\newblock In {\em {Three Courses on Partial Differential Equations}}, pages
  123--162. Berlin: Walter de Gruyter, 2003.

\bibitem{Ki1}
Y.~Kian.
\newblock Equivalence of definitions of solutions for some class of fractional
  diffusion equations.
\newblock {\em Math. Nachr.}, 296(12):5617--5645, 2023.

\bibitem{Kian:2018:TFD}
Y.~Kian, E.~Soccorsi, and M.~Yamamoto.
\newblock On time-fractional diffusion equations with space-dependent variable
  order.
\newblock {\em Ann. Henri Poincar\'{e}}, 19(12):3855--3881, 2018.

\bibitem{KilbasSrivastavaTrujillo:2006}
A.~A. Kilbas, H.~M. Srivastava, and J.~J. Trujillo.
\newblock {\em Theory and {A}pplications of {F}ractional {D}ifferential
  {E}quations}.
\newblock Elsevier Science B.V., Amsterdam, 2006.

\bibitem{Su1}
S.~Kim.
\newblock Unique determination of inhomogeneity in an elliptic equation.
\newblock {\em Inverse Problems}, 18(5):1325--1332, 2002.

\bibitem{Su2}
S.~Kim.
\newblock Recovery of an unknown support of a source term in an elliptic
  equation.
\newblock {\em Inverse Problems}, 20(2):565--574, 2004.

\bibitem{KorabelBarkai:2010}
N.~Korabel and E.~Barkai.
\newblock Paradoxes of subdiffusive infiltration in disordered systems.
\newblock {\em Phys. Rev. Lett.}, 104:170603, 4, 2010.

\bibitem{LiLiuYamamoto:2019}
Z.~Li, Y.~Liu, and M.~Yamamoto.
\newblock Inverse problems of determining parameters of the fractional partial
  differential equations.
\newblock In {\em {Handbook of Fractional Calculus with Applications. {V}ol.
  2}}, pages 431--442. De Gruyter, Berlin, 2019.

\bibitem{LILuY}
Z.~Li, Y.~Luchko, and M.~Yamamoto.
\newblock Analyticity of solutions to a distributed order time-fractional
  diffusion equation and its application to an inverse problem.
\newblock {\em Comput. Math. Appl.}, 73(6):1041--1052, 2017.

\bibitem{LY}
Z.~Li and M.~Yamamoto.
\newblock Uniqueness for inverse problems of determining orders of multi-term
  time-fractional derivatives of diffusion equation.
\newblock {\em Appl. Anal.}, 94(3):570--579, 2015.

\bibitem{LZ}
Z.~Li and Z.~Zhang.
\newblock Unique determination of fractional order and source term in a
  fractional diffusion equation from sparse boundary data.
\newblock {\em Inverse Problems}, 36(11):115013, 20 pp., 2020.

\bibitem{LCY}
H.~Liu, C.-H. Tsou, and W.~Yang.
\newblock On {Calder{\'o}n}'s inverse inclusion problem with smooth shapes by a
  single partial boundary measurement.
\newblock {\em Inverse Problems}, 37(5):055005, 18 pp., 2021.

\bibitem{Orsingher:2018}
E.~Orsingher, C.~Ricciuti, and B.~Toaldo.
\newblock On semi-{M}arkov processes and their {K}olmogorov's
  integro-differential equations.
\newblock {\em J. Funct. Anal.}, 275(4):830--868, 2018.

\bibitem{Rockafellar:1970:ConvexAnalysis}
R.~T. Rockafellar.
\newblock {\em {Convex Analysis}}.
\newblock Princeton University Press, Princeton, NJ, 1970.

\bibitem{RuZ}
W.~Rundell and Z.~Zhang.
\newblock Fractional diffusion: recovering the distributed fractional
  derivative from overposed data.
\newblock {\em Inverse Problems}, 33(3):035008, 27 pp., 2017.

\bibitem{Stein:2003:ComplexAnalysis}
E.~M. Stein and R.~Shakarchi.
\newblock {\em {Complex Analysis}}.
\newblock Princeton University Press, Princeton, NJ, 2003.

\bibitem{Szego:1975:OP}
G.~Szeg\H{o}.
\newblock {\em {Orthogonal Polynomials}}.
\newblock AMS, Providence, RI, fourth edition, 1975.

\bibitem{TrTs}
F.~Triki and C.-H. Tsou.
\newblock Inverse inclusion problem: a stable method to determine disks.
\newblock {\em J. Differential Equations}, 269(4):3259--3281, 2020.

\bibitem{ZhangLiLuo:2013}
H.~Zhang, G.-H. Li, and M.-K. Luo.
\newblock Fractional {F}eynman-{K}ac equation with space-dependent anomalous
  exponent.
\newblock {\em J. Stat. Phys.}, 152(6):1194--1206, 2013.

\end{thebibliography}
\end{document}